\documentclass[11pt]{amsart}
\usepackage{amssymb}
\newtheorem{thm}{Theorem}[section]
\newtheorem{cor}[thm]{Corollary}
\newtheorem{lem}[thm]{Lemma}

\newtheorem{prop}[thm]{Proposition}
\theoremstyle{definition}

\theoremstyle{remark}
\newtheorem{rem}[thm]{Remark}
\numberwithin{equation}{section}

\newcommand{\R}{\mathbb R}

\newcommand{\si}{\sigma}

\newcommand{\la}{\lambda}
\newcommand{\C}{\mathbb C }
\newcommand{\B}{{\mathcal B}}
\newcommand{\pa}{\partial }
\newcommand{\La}{\langle}
\newcommand{\Ra}{\rangle}

\newcommand{\Hc}{\mathcal H}
\newcommand{\sn}{{\mathbf S}^{n-1}}
\newcommand{\sN}{{\mathbf S}^{N-1}}

\newcommand{\om}{ \omega}

\newcommand{\D}{\Delta}
\newcommand{\ap}{\alpha}
\newcommand{\bt}{\beta}
\newcommand{\dl}{\delta}
\newcommand{\N}{\nabla}

\newcommand{\Gm}{\Gamma}
\newcommand{\gm}{\gamma}

\newcommand{\K}{\kappa }
\newcommand{\Dk}{\Delta_{\kappa}}
\newcommand{\pl}{\varphi_{\lambda}}
\begin{document}

\title[]{On the chaotic behavior of the Dunkl heat\\ semigroup
on weighted $ L^p $ spaces}
\author{Pradeep Boggarapu}
\author{S. Thangavelu}

%------------------------------------------------------------------

\address{department of mathematics,
Indian Institute of Science, Bangalore - 560 012, India}
\email{pradeep@math.iisc.ernet.in}
\email{veluma@math.iisc.ernet.in}

%------------------------------------------------------------------


%-------------------------------------------------------------------

\begin{abstract}
In this paper we study the chaotic behaviour of the heat semigroup
generated by the Dunkl-Laplacian on weighted $ L^p$ spaces. In the
case of the heat semigroup associated to the standard Laplacian we
obtain a complete picture on the spaces $ L^p(\R^n,
(\varphi_{i\rho}(x))^2 dx) $ where $ \varphi_{i\rho} $ is the
Euclidean spherical function. The behaviour is very similar to the
case of the Laplace-Beltrami operator on non-compact Riemannian
symmetric spaces studied by Pramanik and Sarkar.
\end{abstract}

% -------------------------------------------------------------------

\maketitle
% --------------------------------------------------------------------

\section{Introduction}

 The study of chaotic dynamics of the heat semigroup on Riemannian
 symmetric spaces of noncompact type, which started with the work of Ji
 and Weber \cite{JW} has been completed recently by
 Pramanik and Sarkar \cite{PS} (see also Sarkar \cite{S}).
 As they have remarked, the chaotic behavior of the
 heat semigroups on $L^p$ spaces seems to be a non-Euclidean
 phenomenon. In order to state the results of Pramanik and Sarkar
 and make a comparison with the Euclidean case, we need to recall
 several definitions from Ergodic theory. We closely follow the
 terminologies used in \cite{PS} referring to \cite{DS} and \cite{JW} for more details.\\

   Let $T_t$, $t>0$ be a strongly continuous semigroup on a Banach
 space $\B$.

\begin{enumerate}
    \item We say that $T_t$ is hypercyclic if there exists a $v\in
    \B$ such that $\{ T_tv : t\geq 0 \}$ is dense in $\B$.

    \item If there exist a $v \in \B$ such that $T_tv=v$ for some
    $t>0$ then we say that $v$ is periodic for $T_t$.

    \item We say that $T_t$ is  chaotic if it is
    hypercyclic and if its periodic points are dense in $\B$.
\end{enumerate}

   Let $\D = -\sum_{j=1}^n \frac{\pa^2}{\pa x_j^2}$ be the standard
 Laplacian on the Euclidean space $\R^n$. The semigroup $T_t=e^{-t\D}$
 generated by $\D$ fails to be chaotic on $L^p(\R^n)$, $1\leq p \leq \infty$.
 This can be easily checked by appealing to the following theorem
 proved in \cite{DE}.
 \begin{thm}[de Laubenfels-Emamirad \cite{DE}]\label{T1}
 If $T_t$ is a chaotic semigroup generated by $A$ in a Banach space
 $\B$ then the cardinality of $\si_{pt}(A)\cap i\R$ is infinite,
 where $\si_{pt}(A)$ is the point spectrum of $A$.
 \end{thm}
 Indeed, the spectrum of $\D $ on $L^p(\R^n)$ is independent of $p$ and
 equals $[0,\infty )$. Consequently, $e^{-t\D}$ cannot be chaotic on any of the $L^p$
 spaces.\\

 Compare this with the case of the heat semigroup generated by the
 Laplace-Beltrami operator $\D_X$ on a noncompact Riemannian
 symmetric space $X$. When $p>2$, there are plenty of
 eigenfunctions, provided by the elementary spherical functions,
 in $L^p(X)$ with purely imaginary eigenvalues. This fact has been
 utilized in obtaining a complete picture of the chaotic behavior
 of $e^{-t\D_X}$ on $L^p(X)$ in the article \cite{PS} where
 the authors have established the following result.\\

\begin{thm} [Pramanik-Sarkar \cite{PS}] \label{T2} For any Riemannian
symmetric space X of non-compact type, let  $T_t $ be the
semigroup generated by $ \Delta_X, T_t^c = e^{ct} T_t $ and let $ c_p =
\frac{4|\rho|^2}{p {p'}}$ where $ \rho $ is the half-sum of the positive roots. Then the following conclusions hold:
(a) For $2<p<\infty, T_t^c $ is chaotic on $L^p(X)$ if and only if
$c>c_p.$ (b) For $p= \infty, T_t^c $ is non-chaotic on $
L^\infty(X)$ for all $ c \in \R.$ (c) However, $T_t^c $ is
subspace-chaotic on $ L^\infty(X) $ if and only if $c > 0.$
\end{thm}

 The analogue of spherical functions in the Euclidean set up are
 the Bessel functions defined by
 $$\varphi_{\la}(x)=\int_{\sn}e^{i\la x\cdot \om}d\si(\om)$$
 where $d\si$ is the surface measure on $\sn$ and $\la\in\C$. These are all
 eigenfunctions of the Laplacian with eigenvalue $\la^2:$
 $\D\varphi_{\la}=\la^2\varphi_{\la}$ and when $\la\in\R$, $\pl\in L^p(\R^n) $,
 $p>\frac{2n}{n-1}$. But when $\la \in \C $, they have exponential growth.
 Indeed,
 $$\pl(x)=c_n\frac{J_{\frac{n}{2}-1}(\la|x|)}{(\la|x|)^{\frac{n}{2}-1}}$$
 where $J_{\ap}(t)$ is the Bessel function of type $\ap$. It follows that
 $$|\pl(x)|\leq  \varphi_{i\Im(\lambda)}(x) = c_n I_{\frac{n}{2}-1}
 (\Im(\la)|x|)(|\Im(\la)|\:|x|)^{-\frac{n}{2}+1}$$
 where $I_{\ap}(t)=J_{\ap}(it)$ is the modified Bessel function. Using the asymptotic
 behavior of $I_{\ap}(t)$ we see that
 $$|\pl(x)|\leq C_\la|x|^{-(\frac{n-1}{2})}e^{|\Im(\la)||x|}.$$
 It then follows that for any $\rho >0,$ and $ p \neq 2,$
 $$\int_{\R^n}|\pl(x)|^p (1+|x|)^{\frac{(n-1)}{2}p} e^{-\rho|p-2|\:|x|}dx < \infty $$
 provided $|\Im(\la)|< \gm_p\rho$ where $\gm_p = \Big |\frac{2}{p}-1\Big |$.
 For $ p > 2,$ we can rewrite the above as
$$ \int_{\R^n} |\varphi_{\lambda}(x) \varphi_{i\rho}(x)^{-1}|^p (\varphi_{i\rho}(x))^2 dx < \infty.$$  \\

It turns out that the functions  $ \varphi_{\lambda}(x)
\varphi_{i\rho}(x)^{-1} $ are eigenfunctions of a modified
Laplacian. Indeed, if we let $ \tilde{\Delta} $ to stand for the
operator defined by $ \tilde{\Delta}f = \varphi_{i\rho}^{-1}
\Delta(f\varphi_{i\rho})$ then clearly $
\tilde{\Delta}(\frac{\varphi_\lambda}{\varphi_{i\rho}}) =
\lambda^2 (\frac{\varphi_\lambda}{\varphi_{i\rho}}).$ From the
very definition of the functions $ \varphi_\lambda $ we have
$$ \varphi_{i\rho}(x) = \int_{S^{n-1}} e^{\rho x \cdot \omega}
d\sigma(\omega) $$ and hence a simple calculation shows that
$$ \tilde{\Delta}f(x) = (\Delta+\rho^2)f(x) - 2 (\varphi_{i\rho}(x))^{-1}\sum_{j=1}^n
\frac{\partial}{\partial x_j}f(x) \frac{\partial}{\partial x_j}
\varphi_{i\rho}(x) .$$ This can be further
simplified by making use of Hecke-Bochner formula for the Fourier
transform. Note that
$$  \frac{\partial}{\partial x_j}\varphi_{i\rho}(x) = \rho \int_{S^{n-1}}
\omega_j e^{\rho x \cdot \omega} d\sigma(\omega) $$ and as $
\omega_j $ are spherical harmonics on $ S^{n-1} $ it follows that
(see Eqn. 15, page 37 in \cite{H})
$$ \int_{S^{n-1}} \omega_j e^{\rho x \cdot \omega} d\sigma(\omega) =
c_n \frac{x_j}{|x|} \frac{I_{n/2}(\rho |x|)}{(\rho |x|)^{n/2-1}}.
$$ Therefore, if we let $ w_n(x) = c_n \frac{I_{n/2}(\rho
|x|)}{I_{n/2-1}(\rho |x|)} $ then
$$ \tilde{\Delta} = \Delta- 2\rho^2 w_n(x) (x \cdot \N) +\rho^2 $$
is a first order perturbation of the Laplacian $ \Delta.$ This
operator $ \tilde{\Delta} $ has eigenfunctions, namely  $
\varphi_\lambda(x) (\varphi_{i\rho}(x))^{-1} $ which belong to $
L^p(\R^n, (\varphi_{i\rho}(x))^2 dx) $ for $ p > 2 $ provided $
|\Im(\lambda)| < \gamma_p \rho.$
It is interesting to note the similarity between $ \tilde{\Delta} $
and the Laplacian $ \Delta_X $ on symmetric spaces.\\

The above suggests that we study the chaotic behaviour of  the
semigroup $ \tilde{T_t} $ generated by $ \tilde{\Delta}+\rho^2 $
on the weighted $ L^p $ spaces
$ L^p(\R^n, (\varphi_{i\rho}(x))^2 dx) .$  This semigroup
is simply obtained from $ T_t = e^{-t(\Delta+\rho^2)} $ by conjugation:
$ \tilde{T_t}(f)(x) = (\varphi_{i\rho}(x))^{-1} T_t(f\varphi_{i\rho})(x).$
In this article our main aim is to prove the following result regarding
the chaotic behaviour of this semigroup.\\

\begin{thm}\label{T3}
 (a) For any $  2 < p < \infty, $ the semigroup $ \tilde{T}_t^c = e^{ct}~ \tilde{T}_t $ is chaotic on
 $ L^p(\R^n, (\varphi_{i\rho}(x))^2 dx) $ if and only if $ c > c_p.$ (b) For
 $ p = \infty, \tilde{T}_t^c $ is not chaotic on $ L^\infty(\R^n) $ for any $ c \in \R.$
(c) For $ 1 \leq p \leq 2,~ \tilde{T}_t^c $ is neither hypercyclic nor
has periodic points on $  L^p(\R^n, (\varphi_{i\rho}(x))^2 dx). $ \\
 \end{thm}

For $ \nu > 0,  z > 0  $ let $ K_\nu(z) $ be the Macdonald
function given by the Sommerfeld integral
$$ K_\nu(z) = \frac{1}{2} \Big (\frac{z}{2}\Big )^\nu \int_0^\infty  e^{-(t+\frac{z^2}{4t})} t^{-\nu-1} dt.$$
By making a change of variables, we observe that
$$  z^\nu K_\nu(z) =  2^{-\nu-1} \int_0^\infty  e^{-(tz^2+\frac{1}{4t})} t^{-\nu-1} dt.$$
The asymptotic behavior of $ K_\nu $ and $ I_\nu(z) $ at infinity
are given by
$$ K_\nu(z)  =
\frac{\sqrt{\pi}}{\sqrt{2z}} e^{-z} (1+O(1/z)), ~~~~~ I_\nu(z) =
\frac{1}{\sqrt{2\pi z}} e^z (1+O(1/z)),$$ see page 226 in
\cite{NU}. Consequently, for $|\Im(\la)|< \gm_p\rho$ we see that
$$\int_{\R^n}|\pl(x)|^p (\tilde{K}_{n/2}( \rho |x|))^ {\gamma_p p} dx < \infty $$
where $ \tilde{K}_\nu(z) = z^\nu K_\nu(z).$ If we take $ \lambda =
\beta(1+i), |\beta| <  \gamma_p  \rho,$ then  $\D\pl=2i\bt^2\pl $
and hence we have plenty of eigenfunctions with purely imaginary
 eigenvalues which belong to the weighted $L^p$ spaces
 $L^p(\R^n, (\tilde{K}_{n/2}( \rho |x|))^ {\gamma_p p}dx)$.\\

The behaviour of the modified semigroup $ \tilde{T_t} $ on $
L^p(\R^n, (\varphi_{i\rho}(x))^2 dx) $ is equivalent to the
behaviour of $ T_t $ on the spaces $L^p(\R^n, (\tilde{K}_{n/2}(
\rho |x|))^ {\gamma_p p}dx)$ for $ p > 2.$
 Indeed, if we set $ \tilde{I}_\nu(z) = \frac{I_\nu(z)}{z^\nu},$ then
it follows from the asymptotic properties of $ I_\nu $ and $ K_\nu
$ that
\begin{equation}\label{J}
   C_1 \leq \tilde{I}_\nu(z) \tilde{K}_{\nu+1}(z) \leq C_2 ,~~~~ z \geq 0.
\end{equation}
In view of this, $L^p(\R^n, (\tilde{K}_{n/2}( \rho |x|))^
{p-2}dx)$ is the same as $L^p(\R^n, (\varphi_{i\rho}(x))^
{2-p}dx)$ whenever $ p > 2$ as $ \gamma_p p = p-2.$  It follows
that $ f \in L^p(\R^n, (\varphi_{i\rho}(x))^2 dx) $ if and only if
$ f \varphi_{i\rho} \in L^p(\R^n, (\tilde{K}_{n/2}( \rho |x|))^
{p-2}dx).$ It is
 therefore, natural to study the heat semigroup $e^{-t(\D+\rho^2)}$ on the
 weighted $L^p$ spaces $L^p(\R^n, (\tilde{K}_{n/2}( \rho |x|))^ {\gamma_p p}dx)$.
 It turns out that the chaotic behavior of
 $e^{-t(\D+\rho^2)}$ on these spaces is very similar to the behavior of
 $e^{-t\D_X}$ on $L^p(X)$, $p>2$. Indeed, we have the following theorem. \\

In what follows we write $ L^p_\rho(\R^n) $ in place of $
L^p(\R^n, (\tilde{K}_{n/2}( \rho |x|))^ {\gamma_p p}dx)$ for the
sake of notational convenience.
\begin{thm}\label{T4}
 For $1\leq p \leq \infty$ let $c_p=\rho^2(1-\gm_p^2)$
 and  for $c\in \R$ define $T_t^c=e^{-t(\D+\rho^2-c)}$  where $ \Delta $
 is the standard Laplacian on $ \R^n.$ Then
 \begin{enumerate}
    \item For $1 \leq p < \infty, p\neq 2, ~T_t^c$ is chaotic on $L^p_{\rho}(\R^n)$ if $c>c_p.$
    \item $T_t^c$ is not chaotic on $L^{\infty}_{\rho}(\R^n)$ for any $c\in \R $.
    \item For  $1\leq p < \infty$ and $c \leq c_p$, $T_t^c$ is not chaotic  on
    $L^p_{\rho}(\R^n)$ and for $ c < c_p $ it is not even hypercyclic.
 \end{enumerate}
 \end{thm}
The proof of the above theorem depends on  a sharp estimate
 for the heat semigroup $ T_t $ on $ L^p_{\rho}(\R^n)$
 stated and proved in Proposition 3.4.  And
 Theorem 1.3 is an immediate consequence of  the above result.\\

 In this paper we also  work in a more general set up and study the
 chaotic dynamics of the heat semigroup generated by the Dunkl
 Laplacian $\Dk$ on $\R^n$ associated to a finite reflection group.
 Let $G$ be such a group generated by the reflections associated
 to a root system on $\R^n$. Let $\K$ be a nonnegative multiplicity
 function and $h_{\K}^2(x)$ be the associated weight function.
 Let $T_j$, $j=1,2,\cdots ,n$ be the Dunkl difference-differential
 operators and $\Dk= -\sum_{j=1}^n T_j^2$ be the Dunkl Laplacian.
 For all the required definitions we refer to Section 3. We consider
 $T_t=e^{-t(\Dk +\rho^2)}$  the heat semigroup generated by
 $ A = \Dk+\rho^2$ on the weighted $ L^p-$space
 $L^p(\R^n, (\tilde{K}_{n/2+\gamma}( \rho |x|))^ {\gamma_p p} h_{\K}^2(x)dx)$,
where $ \gamma $ is defined in terms of the multiplicity function
$ \K ,$ see Section 3. We denote the above space by
$L_{\rho,\K}^p(\R^n).$
 Note that $L^p_{\rho, 0}(\R^n) = L^p(\R^n, (\tilde{K}_{n/2}( \rho |x|))^{\gamma_p p}dx)
 $ which we have denoted by $L^p_{\rho}(\R^n)$.
 The chaotic behavior of the semigroup generated by $A$ is
 described in the following result. For $1\leq p \leq \infty$
 we denote $p'$ to be conjugate index of $p$ i.e. $1/p+1/p'
 =1$.\\

 \begin{thm}\label{T5}
Let $ A = \Delta_\K+\rho^2.$ For $1\leq p \leq \infty$, let $c_p=\rho^2(1-\gm_p^2)$
 and  for $c\in \R$ define $T_t^c=e^{-t(A-c)}.$ Then
 \begin{enumerate}
    \item For $1\leq p <\infty,~ p\neq 2 $, $T_t^c$ is chaotic on $ L_{\rho,\K}^p(\R^n)$ if and only if $c>c_p.$
    \item $T_t^c$ is not chaotic on $L_{\rho,\K}^\infty(\R^n)$ for any $c\in \R $.
    \item $T_t^c $ is not chaotic on $ L^2_{\rho,\K}(\R^n) = L^2(\R^n, h_\K(x)^2 dx) $ for any $ c \in \R^n.$
 \end{enumerate}
 \end{thm}

For $ c \leq c_p, $ we would like to know  which property of  $ T_t^c $ fails. The next theorem partially answers this question.

 \begin{thm}\label{T5.1}
With same notations as in the previous theorem we have the following results.
 \begin{enumerate}

    \item For  $1\leq p < 2$   and $c< \frac{2\rho^2}{p'}$,
     $T_t^c$ is not hypercyclic on $L^p_{\rho,\K}(\R^n).$
    \item For  $p>2$ and
  $c< \frac{2\rho^2}{p}$, $T_t^c$ is not hypercyclic on
    $L^p_{\rho,\K}(\R^n).$
 \end{enumerate}
 \end{thm}

 \begin{rem}\label{R1}
    When $1\leq p < 2$ (when  $p>2$) and $\frac{2\rho^2}{p'}\leq c \leq c_p $
    (resp. when  $\frac{2\rho^2}{p}\leq c \leq c_p$), we don't know if $T_t^c$ fails to be hypercyclic or not.
    Also we are not able to say anything about the periodicity. This is due to the fact that we do not have
    sharp estimates on the operator norm of $ T_t $ on $ L^p_{\rho,\K}(\R^n).$
    On the other hand when $ \K =0 $ we do have better estimates
    for the operator norm of $ T_t $ and hence we have a complete picture.
 \end{rem}

 An examination of the proof of the sharp estimate for
 $ T_t =e^{-t(\D+\rho^2)}$ in Theorem \ref{T7} reveals that
 we need to use the boundedness of translation operators on
 $ \R^n$ on weighted $ L^p $ spaces. If we want to prove an
 analogue of Theorem \ref{T7} for the Dunkl Laplacian, then
 we need to know the boundedness properties of Dunkl translation
 on the spaces $ L_{\rho,\K}^p(\R^n).$ Unfortunately,
 the boundedness properties of these operators are not even
 known on $ L^p $ spaces, see \cite{TX} for some results.\\

On the other hand, instead of $ L_{\rho,\K}^p(\R^n)$ for $1\leq p
< \infty, $ if we consider the mixed norm spaces $
L_{\rho,\K}^{p,2}(\R^n),$ then we can improve the estimates. These
are defined as the space of all functions $ f $ for which
$$ \int_0^{\infty}
\left( \int_{\sn}|f(r\om)|^2h_{\K}^2(\om)d\si(\om) \right)
 ^{\frac{p}{2}}(\tilde{K}_{n/2+\gamma}(\rho r))^{p \gamma_p}r^{n+2\gm-1}dr  < \infty .$$
The $p$-th root of the above quantity will be denoted by $
\|f\|_{L_{\rho,\K}^{p,2}(\R^n)}.$ On this space  we have better 
estimates for  the Dunkl heat semigroup, see Theorem \ref{T14}.
Consequently, we can prove the following result.\\

 \begin{thm}\label{T6}
 For $1\leq p < \infty$ let $c_p=\rho^2(1-\gm_p^2)$
 and   for $c\in \R$ define $T_t^c =e^{-t(A-c)}.$  Then
 \begin{enumerate}
    \item For $1\leq  p <\infty,~ p\neq 2 $, $T_t^c$ is chaotic on $L^{p,2}_{\rho,\K}(\R^n)$ if $c>c_p.$
    \item For  $1\leq p < \infty$ and $c< c_p$, $T_t^c$ is not hypercyclic on
    $L^{p,2}_{\rho,\K}(\R^n)$ and hence not chaotic.
 \end{enumerate}
 \end{thm}

 A comparison of these theorems with the results of \cite{PS} (see Theorems
 1.2, 1.3 and 1.4) shows the similarity between the behavior of
 $e^{-t(A-c)}$ on $L^p_{\rho,\K}(\R^n) $ and $e^{-t(\D_X-c)}$ on $L^p(X)$.
 It is also interesting to compare our results to the unweighted
 case of the heat semigroup $e^{-t(-\D-c)}$ on $L^p(\R^n) $
 stated and proved in Section 9 of \cite{PS}.

%--------------------------------------------------------------------------------------------

\section{Chaotic behavior of the heat semigroup\\ on weighted $ L^p$ spaces}
\subsection{ The heat semigroup on $L^p_\rho(\R^n)$:} In this subsection
we will estimate the operator norm of $ T_t $ acting on the weighted space
$ L^p_\rho(\R^n).$ In proving the following result we will make use of the
fact that for $ 1< p <\infty,$ the dual of $ L^p_\rho(\R^n) $ can be identified with
$ L^{p'}(\R^n, (\varphi_{i\rho}(x))^{p' \gamma_{p'}} dx)$ if the duality bracket is taken as
$$ (f,g) = \int_{\R^n} f(x) g(x) dx,$$
for $ f \in L^p_\rho(\R^n), g \in L^{p'}(\R^n,
(\varphi_{i\rho}(x))^{p' \gamma_{p'}} dx).$ This follows from the
estimates \eqref{J}. When $ p =1 $ the dual of $ L^1_\rho(\R^n) $
is taken as $ L^\infty(\R^n) $ and we use the standard duality
bracket
$$ (f,g) = \int_{\R^n} f(x) g(x)  \tilde{K}_{n/2}(x) dx.$$

 \begin{thm}\label{T7}
 Let $T_t$ be the semigroup generated by $\D+\rho^2.$ Then for any
 $1\leq p < \infty$ it is strongly continuous on $ L^p_\rho(\R^n).$
Moreover, we have the estimate
\begin{eqnarray*}
   \|T_tf(x)\|_{L^p_\rho(\R^n)}\leq C (1+t)^{\frac{n-1}{2}(1+\gm_p)} e^{-\frac{4\rho^2}{p p'}t}
 \|f(x)\|_{L^p_\rho(\R^n)}
 \end{eqnarray*}
for all $ f \in L^p_\rho(\R^n)$ and $1 \leq p < \infty .$
 \end{thm}
\begin{proof}  In view of the asymptotic behavior of the Macdonald function,
it is enough to consider the space defined using $
(1+|x|)^{(n-1)/2}e^{-\rho |x|} $ in place of $
\tilde{K}_{n/2}(\rho |x|).$ For the sake of brevity, just for this section,
we denote the weight function $
(1+|x|)^{(n-1)/2}e^{-\rho |x|}$ by $w_{\rho}(x).$  It is
therefore enough to prove
 \begin{eqnarray*}
   & &\left(\int_{\R^n}|T_tf(x)|^p (w_{\rho}(x))^{p\gm_p} dx
   \right)^{\frac{1}{p}}\\
 & &\leq C (1+t)^{\frac{n-1}{2}(1+\gm_p)} e^{-\frac{4\rho^2}{p p'}t}
 \left(\int_{\R^n}|f(x)|^p (w_{\rho}(x))^{p\gm_p} dx \right)^{\frac{1}{p}}.
 \end{eqnarray*}
The strong continuity of $T_t$ on $L^p_{\rho}(\R^n)$ follows from
the norm estimates. Indeed, for $0<t\leq 1$, the operators $T_t$
are uniformly bounded on
 $L^p_{\rho}(\R^n)$, $1\leq p <\infty$. As $L^p(\R^n, dx)$
 is dense in $L^p_{\rho}(\R^n)$, the strong continuity of $T_t$
 on $L^p_{\rho}(\R^n)$ follows from the same on $L^p(\R^n, dx)$ in view of
 the continuous inclusion $L^p(\R^n, dx)\subset L^p_{\rho}(\R^n)$.\\

 First assume that $ 1 < p < \infty.$ For any $f \in L^p_{\rho}(\R^n)$ and
 $g \in (L^{p}_{\rho}(\R^n))^*\simeq L^{p'}(\R^n, (\varphi_{i \rho}(x))^{\gamma_{p'} p'} dx)$ consider
 $$\int_{\R^n}T_tf(x)g(x) dx = e^{-t\rho^2}\int_{\R^n}\int_{\R^n}f(y)h_t(x-y)g(x)dx dy $$
where $ h_t $ is the heat kernel which is explicitly given by
$$ h_t(x) = (4\pi t)^{-n/2} e^{-\frac{1}{4t}|x|^2}.$$
By making a change of variables  in the  $x$-integral,  the above reads as
 $$ e^{-t\rho^2}\int_{\R^n}h_t(x)\Big (\int_{\R^n}f(y)g(x+y)dy \Big )dx.$$
 Now the inner integral can be estimated by H\"{o}lder's inequality after rewriting it as
 $$\int_{\R^n}f(y)(w_{\rho}(y))^{\gm_p}g(x+y)(w_{\rho}(y))^{-\gm_p}dy.$$
 Since $ \gamma_p = \gamma_{p'} $ the result is the bound
 $$\Big (\int_{\R^n}|f(y)|^p (w_{\rho}(y))^{p\gm_p}dy \Big )^{\frac{1}{p}}
 \Big (\int_{\R^n}|g(x+y)|^{p'} (w_{\rho}(y))^{-p'\gm_{p'}}dy \Big )^{\frac{1}{p'}}. $$
 By making a change of variables, the second integral can be written as
 $$ \Big (\int_{\R^n}|g(y)|^{p'} (w_{\rho}(y-x))^{-p'\gm_{p'}}dy \Big )^{\frac{1}{p'}}. $$
 In view of the inequality $(1+|y|)\leq (1+|x-y|)(1+|x|)$,
 we see that $(w_{\rho}(y-x))^{-\gm_{p'}p'}\leq (1+|x|)^{\frac{n-1}{2}p'\gm_{p'}}
 e^{ p'\gm_{p'}\rho|x|}(w_{\rho}(y))^{-p'\gm_{p'}}$. By making use
 of this, the above integral is bounded by
 $$(1+|x|)^{\frac{n-1}{2}\gm_{p}} e^{\gm_{p'}\rho|x|}\Big (\int_{\R^n}|g(y)|^{p'}
 (\varphi_{i\rho}(y))^{p'\gm_{p'}}dy \Big )^{\frac{1}{p'}} .$$
 Thus $ \left|\int_{\R^n}T_t f(x) g(x) dx\right| $ is bounded by
$$ e^{-t\rho^2}\|f\|_{L^p_{\rho}(\R^n)}\|g\|_{(L^{p}_{\rho}(\R^n))^*}
   \int_{\R^n}h_t(x)(1+|x|)^{\frac{n-1}{2}\gm_{p}} e^{\gamma_p \rho |x|} dx.$$
A  simple calculation shows that
 $$ e^{-t\rho^2} \int_{\R^n}h_t(x)(1+|x|)^{\frac{n-1}{2}\gm_{p'}}e^{ \gm_{p'} \rho|x|}dx
 \leq C(1+t)^{\frac{n-1}{2}(1+\gm_p)}e^{-t\rho^2 (1-\gm_p^2)}  $$
which  completes the proof as $1-\gm_p^2 = \frac{4}{p p'}.$ When $ p = 1 $ we can directly estimate
$$  \int_{\R^n} \int_{\R^n} h_t(x-y) |f(y)| (1+|x|)^{(n-1)/2} e^{-\rho |x|} dx dy.$$
Just make a change of variables in the $x$-integral and proceed as before to get the required estimate.
\end{proof}

%-----------------------------------------------------------------------------------

\subsection{Spectrum of the Laplacian on weighted $L^p$ spaces}
 In case of noncompact Riemannian symmetric spaces $G/K$ the $L^p$
 spectrum of the Laplace-Beltrami operator $\D$ is precisely known.
 It has been proved in Taylor \cite{TM} that the $L^p$ spectrum is
 equal to the parabolic neighborhood
 $$\mathfrak{P}_p=\{ \la^2+|\rho|^2: | \Im(\la)|\leq \gm_p|\rho| \}  $$
 of the half line $[|\rho|^2, \infty ) $. This follows from a
 multiplier theorem proved in \cite{TM} for general Riemannian
 manifolds. It would be nice to see if we have precise
 information about the spectrum of the Dunkl-Laplacian $ \D_\K$
 acting on the spaces $L^p_{\rho, \K}(\R^n)$. In this
 generality, we are not able to determine precisely the spectrum
 of $\D_\K$. However, when $\K=0$ i.e. for the standard Laplacian $\D$ on
 $\R^n$ we do have the following result. In view of the asymptotic behavior of the Macdonald function,
it is enough to consider the space defined using
$w_{\rho}(x):=(1+|x|)^{(n-1)/2}e^{-\rho |x|} $ in place of $
\tilde{K}_{n/2}(\rho |x|).$
 \begin{thm}\label{T8}
 For any $1\leq p<\infty $, the spectrum of $ \D+\rho^2$ on
 $L^p_{\rho}(\R^n)$ is precisely the set
 $$ \mathfrak{P}_p = \{ \la^2+\rho^2: | \Im(\la)|\leq \gm_p \rho \} .$$
For $ p > 2,$ the spectrum of $ \tilde{\Delta} $ on $ L^p(\R^n,(\varphi_{i\rho}(x))^2 dx) $ is also $\mathfrak{P}_p.$
 \end{thm}
 As in the case of symmetric spaces this result can be deduced
 from the following multiplier theorem for the Laplacian on the
 weighted $ L^p$ spaces $L^p_{\rho}(\R^n) $.\\

 In order to state the result we recall some definitions from
 \cite{TM}. Let $ \Omega_W$ be the set $\{\la \in \C:~ |\Im(\la)|<W\}$
 and set $\mathfrak{S}_W^m$ to be the set of all
 even holomorphic functions $ \varphi $ on $\Omega_W$ satisfying
 $$|\varphi^{(j)}(\la) | \leq C_j (1+\la^2)^{\frac{m-j}{2}}$$
 on the closure $ \overline{\Omega}_W$ for all $j=0,1,2,\cdots  $.
 \begin{thm}\label{T9}
 For every $1<p<\infty $ we have $\varphi(\sqrt{\D}):
 L^p_{\rho}(\R^n)\rightarrow L^p_{\rho}(\R^n)$ provided
 $\varphi \in \mathfrak{S}_W^0 $ with $ W \geq \gm_p \rho $.
 \end{thm}
 In proving this theorem we closely follow \cite{TM} (see proof of Theorem
 A). We use the functional calculus to write
 $$\varphi(\sqrt{\D})=(2\pi)^{-\frac{1}{2}}\int_{-\infty}^{\infty}
 \widehat{\varphi}(t)\cos{t\sqrt{\D}}\: dt .$$
 Using a partition of unity we write $\widehat{\varphi}=
 \widehat{\varphi}_1 + \widehat{\varphi}_2 $ where
 $\widehat{\varphi}_1$ is compactly supported and
 $\widehat{\varphi}_2(t)=0$ for $|t|$ small. As a consequence of this
 decomposition we have
 \begin{lem}[see Lemma 1.3 in \cite{TM}]\label{L1}
 Given $\varphi \in \mathfrak{S}_W^m$ we can write
 $\varphi= \varphi_1+\varphi_2$ where $\widehat{\varphi}_1$
 has compact support, $\varphi_1\in \mathfrak{S}_{W'}^m$ for
 all $W' < W $ and $\varphi_2 \in \mathfrak{S}_W^m $.
 \end{lem}
 In an earlier paper \cite{CGT} it has been proved that
 $\varphi_1(\sqrt{\D})$ is a pseudo-differential
 operator whose distribution kernel is supported
 near the diagonal. Consequently, the boundedness of
 pseudo differential operators of order 0 on $L^p$ spaces gives us

 \begin{lem}\label{L2}
 If $\varphi_1$ is as in the previous lemma with $m=0$
 then for any $1<p<\infty$, $\varphi_1(\sqrt{\D}) $
 is bounded on $ L^p(\R^n)$.
 \end{lem}
 Using the fact that the distribution kernel of
 $\varphi_1(\sqrt{\D})$ is supported in a neighborhood
 of the diagonal, say $|x-y|\leq \frac{1}{2}$, we can actually
 prove the boundedness of $\varphi_1(\sqrt{\D})$ on the
 weighted $L^p$ spaces $L^p_{\rho}(\R^n)$. To see this let
 $k(x,y)$ be the distribution kernel of $\varphi_1(\sqrt{\D})$
 which is supported in $|x-y|\leq \frac{1}{2}$ so that
 $$\varphi_1(\sqrt{\D})f(x)=\int_{|x-y|\leq \frac{1}{2}}k(x,y)f(y)dy.$$
 Consider now $\int_{\R^n}|\varphi_1(\sqrt{\D})f(x)|^p(w_{\rho}(x))^{p\gm_p}dx
 $ which is equal to
 \begin{eqnarray*}
    & &
   \sum_{m=0}^{\infty}\int_{m-\frac{1}{2}\leq |x| < m+\frac{1}{2}}
   |\varphi_1(\sqrt{\D})f(x)|^p(w_{\rho}(x))^{p\gm_p}dx  \\
   & \leq & C \sum_{m=0}^{\infty} (m+2)^{\frac{(n-1)}{2}p\gm_p}e^{-p\gm_p\rho (m-\frac{1}{2})}
   \int_{m-\frac{1}{2}\leq |x| < m+\frac{1}{2}}|\varphi_1(\sqrt{\D})f(x)|^pdx.
 \end{eqnarray*}
 Since the kernel $k(x,y)$ is supported in $|x-y|\leq\frac{1}{2}$
 we observe that
 \begin{eqnarray*}
    \chi_{m-\frac{1}{2}\leq |x| < m+\frac{1}{2}}(x)\varphi_1(\sqrt{\D})f(x)
    = \int_{|x-y|\leq\frac{1}{2}}k(x,y)f(y)\chi_{m-1\leq |y| \leq m+1}(y) dy.
 \end{eqnarray*}
 Consequently, the boundedness of $ \varphi_1(\sqrt{\D}) $ gives the estimate
 \begin{eqnarray*}
    \int_{m-\frac{1}{2}\leq|x|< m+\frac{1}{2}}|\varphi_1(\sqrt{\D})f(x)|^p
    dx
    \leq  C \int_{m-1\leq |y| \leq m+1}|f(y)|^p dy
    \end{eqnarray*}
which is easily seen to be bounded by $$
C (1+m)^{-\frac{(n-1)}{2}p\gm_p}e^{p\gm_p \rho(m+1)}\int_{m-1\leq |y| <m+1}|f(y)|^p
    (w_{\rho}(y))^{p\gm_p}dy.$$
 Summing over $m$ we obtain
 \begin{eqnarray*}
     \int_{\R^n}|\varphi_1(\sqrt{\D})f(x)|^p(w_{\rho}(x))^{p\gm_p}dx
    \leq C \int_{\R^n}|f(y)|^p(w_{\rho}(y))^{p\gm_p}dy
 \end{eqnarray*}
 which takes care of $\varphi_1(\sqrt{\D})$. The proof of Theorem \ref{T9} will be complete once we prove the following
 result.\\
 \begin{thm}\label{T10}
 If $\varphi \in \mathfrak{S}_W^{-\infty}$, $W\geq \gm_p\rho$
 then $$\varphi(\sqrt{\D}):L^p_{\rho}(\R^n)\rightarrow L^p_{\rho}(\R^n)$$
 is bounded for $1\leq p<\infty $.
 \end{thm}
 \begin{proof}
  Once again the proof is a modification of the proof of Proposition 1.4
  in \cite{TM}. We only prove the theorem when $1\leq p\leq 2$.
  The case $p>2 $ can be  handled by duality. Since $\varphi(\sqrt{\D})$
  is bounded on $L^2(\R^n)$, it is enough to
  prove the boundedness of $\varphi(\sqrt{\D})$ on $L^1_{\rho}(\R^n)=
  L^1(\R^n, \tilde{K}_{n/2}(\rho |x|)dx)$. For then, we can appeal to
  Stein-Weiss interpolation theorem with change of measures to get
  the desired result, see \cite{TM}. In fact our choice of the measure,
  namely $(\tilde{K}_{n/2}(\rho |x|))^{p \gamma_p} $ is motivated by this theorem.\\

  If $k_\varphi(x,y)$ stands for the kernel of $\varphi(\sqrt{\D}),$
  we need to show that
  $$\sup_{y\in \R^n}(w_{\rho}(y))^{-1}\int_{\R^n}|k_{\varphi}(x,y)|w_{\rho}(x)dx \leq C. $$
  Let $A_y(m)$ be the annulus $\{x:~m\leq |x-y|< m+1\}$ and consider
  $$\int_{\R^n}|k_{\varphi}(x,y)|w_{\rho}(x)dx=\sum_{m=0}^{\infty}
  \int_{A_y(m)}|k_{\varphi}(x,y)|w_{\rho}(x)dx. $$
  By Cauchy-Schwarz we estimate the above by
  $$\sum_{m=0}^{\infty}\Big ( \int_{A_y(m)}(w_{\rho}(x))^2dx\Big )^{\frac{1}{2}}
  \Big ( \int_{A_y(m)}|k_{\varphi}(x,y)|^2 dx\Big )^{\frac{1}{2}} .$$
  Now\begin{eqnarray*}
   (w_{\rho}(y))^{-2}\int_{A_y(m)}(w_{\rho}(x))^2dx &=&\int_{A_y(m)}\Big ( \frac{1+|x|}{1+|y|}\Big
  )^{(n-1)}e^{2\rho(|y|-|x|)}dx\\
     &\leq& (m+1)^{n-1} e^{2\rho(m+1)}\int_{A_y(m)}dx\\
     &\leq& C(m+1)^{2(n-1)} e^{2m\rho}
  \end{eqnarray*}
  and consequently
  $$(w_{\rho}(y))^{-1}\int_{\R^n}|k_{\varphi}(x,y)|w_{\rho}(x)dx$$
  $$\leq C \sum_{m=0}^{\infty}(m+1)^{n-1}e^{m\rho}\Big (
  \int_{A_y(m)}|k_{\varphi}(x,y)|^2 dx \Big )^{\frac{1}{2}}.$$
  Let $L^2$ norm of $k_{\varphi}(x,y)$ over the annulus $A_y(m)$
  can be estimated as in \cite{TM}. For the convenience of
  the reader we give some details.\\

  Let $\dl_y$ stand for the Dirac delta distribution at $y$. Then we
  can find functions $g_y$ and $h_y$ both in $L^2(|x-y|\leq 1)$
  such that $\dl_y=\D^{\frac{s}{2}}g_y +h_y$ where
  $s=\Big [ \frac{n}{4}\Big]+1$. We can assume that
  $\|g_y\|_2$ and $\|h_y\|_2$ are bounded uniformly in $y$.
  With this decomposition of $\dl_y$ we obtain
  $$k_{\varphi}(x,y)=\varphi(\sqrt{\D})\dl_y(x)=
  \varphi_s(\sqrt{\D})g_y(x)+\varphi(\sqrt{\D})h_y(x)$$
  where $\varphi_s(\la)=\la^{2s}\varphi(\la)$.
  By the finite propagation speed, on the annulus $A_y(m)$
  we have
  $$\varphi_s(\sqrt{\D})g_y(x)=\int_{|t|\geq m-1}\widehat{\varphi}_s(t)(\cos{t\sqrt{\D}})g_y(x)dt. $$
  Since $\varphi_s \in \mathfrak{S}_W^{-\infty}$, we have the estimate
  $$|\widehat{\varphi}_s(t)|\leq C_N (1+t^2)^{-\frac{N}{2}}e^{-W|t|},~ ~ N=0,1,2,\cdots .$$
  Using the boundedness of $\cos{t\sqrt{\D}} $ on $L^2(\R^n)$ we have, for any $N$,
  $$\Big (\int_{A_y(m)}|\varphi_s(\sqrt{\D})g_y(x)|^2 dx \Big )^{\frac{1}{2}}
  \leq C_N \int_{|t|\geq m-1}(1+t^2)^{-\frac{N+2}{2}}e^{-W|t|}dt$$
  $$\leq C_N(1+m^2)^{-\frac{N}{2}}e^{-Wm} .$$
  A similar estimate holds for $\varphi(\sqrt{\D})h_y$ on $A_y(m)$.
  Putting these estimates together we have
  $$ (w_\rho(y))^{-1} \int_{\R^n}|k_{\varphi}(x,y)|w_\rho(x) dx $$
  $$\leq C \sum_{m=0}^{\infty}(m+1)^{n-1}e^{m\rho}(1+m^2)^{-\frac{N}{2}}e^{-Wm} $$
  which is finite provided $W\geq \rho $ if we take $N> n+1 $.
  This completes the proof of Theorem \ref{T10}.
 \end{proof}

%--------------------------------------------------------------------------

\subsection{The chaotic behavior of the heat semigroup:}
   In this subsection we prove one of the main theorems namely Theorem \ref{T4} regarding the chaotic
 behavior of the semigroup $T_t^c = e^{ct}e^{-t(-\D+\rho^2)}$,
 $c \in \R$ on the space $L_{\rho}^p(\R^n)$. In proving the
 result we closely follow the proofs given in \cite{PS} for the case of
 symmetric spaces of non-compact type. As in \cite{PS} we let
 $$\Lambda_p = \{\la \in \C: \; |\Im(\la)|\leq \gm_p\rho \} $$
 and define $c_p= (1-\gm_p^2)\rho^2=\frac{4\rho^2}{pp'}$.
 For any $c>c_p$ we also define a map
 $$\om_c:\Lambda_p \rightarrow \C, \; \; \om_c(\la)= \la^2+\rho^2-c.$$
 Using this we define the following three subsets of $\Lambda_p^0$,
 the interior of $\Lambda_p$;
 \begin{eqnarray*}
   A_1 &=& \{\la \in \Lambda_p^0: \Re(\om_c(\la))>0\}; \\
   A_2 &=& \{\la \in \Lambda_p^0: \Re(\om_c(\la))<0\}; \\
   A_3 &=& \{\la \in \Lambda_p^0: \om_c(\la) \in i\mathbb{Q}\};
 \end{eqnarray*}
 where $\mathbb{Q}$ is the set of all rationals. In \cite{PS} the authors
 have proved that all these sets are non-empty and $A_3$ has infinitely
 many points (see Lemma 4.1 in \cite{PS})  for $ c > c_p$.
 Also note that $A_1$ and $A_2$ are both open subsets of $\Lambda_p^0$. \\

  To each of these $A_j$'s we associate certain subsets $\mathcal{A}_j$ as
  follows. The translation of the spherical functions $\varphi_{\la}(y)$
  are given  by the equation
  $$\tau_x\varphi_\la(y)= \int_{\sn}e^{i\la (x+y)\cdot \om}d\si(\om).$$
  It is therefore clear that $\tau_x\varphi_\la$ are eigenfunctions of $\D + \rho^2$ with eigenvalues
  $\la^2+\rho^2$ and that $\tau_x\varphi_\la \in L^p_{\rho}(\R^n)$ whenever $\la \in \Lambda_p^0$ and $p\neq 2 $.
  This follows from the estimate $|e^{i\la x\cdot \om}|\leq C e^{|\Im(\la)|\:|x|}$. For each $1\leq j \leq 3 $ we set
  $$\mathcal{A}_j= \{\tau_x\varphi_\la(x): \; x \in \R^n, \; \la \in A_j\} .$$
  It is then clear that $\mathcal{A}_j\subset L^p_{\rho}(\R^n)$
  and that for $f \in \mathcal{A}_j$, $T_tf = e^{-t\om_c(\la)}f $.\\

  We now recall certain results from the general theory
 of chaotic semigroups. Given a strongly continuous semigroup
 $T_t$ on a Banach space $\B$ the following three subsets of $\B$
 are important in detecting the chaotic behavior of $T_t$:
 \begin{eqnarray*}
   \mathcal{B}_0 &=& \{x \in \B: \; \lim_{t \rightarrow \infty}T_t x =0 \};  \\
   \B_{\infty} &=& \{x \in \B: \; \forall \epsilon>0 \; \exists w \in \B \text{ and } t>0
   \text { such that  } \| w\|< \epsilon \\ & & \text { and } \|T_tw - x\|< \epsilon \}; \\
   \B_{Per} &=& \text{ the set of all periodic points of } T_t .
 \end{eqnarray*}

\begin{thm}[see \cite{DS}, Theorem 2.3]\label{T11} Let $\B$ be a separable Banach space at let
$T_t$ be a strongly continuous semigroup on $\B$. If both $\B_0$
and $\B_{\infty}$ are dense in $\B$ then $T_t$ is hypercyclic.
\end{thm}
\begin{cor}\label{C1} Let $\B$ and $T_t$ be as above. If all $\B_0$, $\B_{\infty}$ and
$\B_{Per}$ are dense in $\B$ then $T_t$ is chaotic.
\end{cor}
Once we have the above Theorem \ref{T11} and Corollary \ref{C1},
the sufficiency part of the Theorem \ref{T4} follows immediately
from the next proposition. Recall that we are considering the
Banach spaces $ L^p_\rho(\R^n) $ and the semigroup $ T_t^c =
e^{-t(\Delta+\rho^2-c)}.$

\begin{prop}\label{P1} For each $1\leq j \leq 3 $,  $\mathcal{A}_j$ is dense in $L^p_{\rho}(\R^n)$,
$1 \leq  p<\infty,~~ p\neq 2 $ and $span(\mathcal{A}_1)\subset
\B_0$, $span(\mathcal{A}_2)\subset \B_\infty$ and
$span(\mathcal{A}_3)\subset \B_{Per}$ provided $c>c_p$.
\end{prop}

\begin{proof} We first prove the set inclusions $span(\mathcal{A}_1)\subset \B_0$,
$span(\mathcal{A}_2)\subset \B_\infty$ and
$span(\mathcal{A}_3)\subset \B_{Per}$. When $\la \in A_1$,
$\Re(\om_c(\la))>0$ and hence for any $f \in \:
span(\mathcal{A}_1)$, $T_t^cf = e^{-t\om_c(\la)}f $ from which it is
clear that $\lim_{t\rightarrow \infty}T_t^cf=0$. If $g \in \:
span(\mathcal{A}_2)$, then $g=\sum_{k=1}^m
a_k\tau_{x_k}\varphi_{\la_{k}}$ with $\la_k \in A_2$.
Consequently, if we define $f_t =\sum_{k=1}^m
a_ke^{t\om_c(\la_{k})} \tau_{x_k}\varphi_{\la_{k}}$ then it
follows that $T_t^cf_t = g$ for any $t>0$. Since
$\Re(\om_c(\la_k))<0$, for any $\epsilon>0$ we can choose $t$
large enough so that $\|f_t\|_{L^p_{\rho}(\R^n)} < \epsilon $. On
the other hand $\|T_t^cf_t -g \|_{L^p_{\rho}(\R^n)}=0 $ and hence $g
\in \B_{\infty}$. The proof of third inclusion is also easy. As it
is similar to the case of symmetric spaces we leave the proof and
refer to \cite{PS}. Now we need to prove the density of $\B_0,\;
\B_\infty $ and $\B_{Per}$ which will follow once we prove that of
$\mathcal{A}_j$ in $L^p_{\rho}(\R^n)$. Suppose the span of
$\mathcal{A}_1$ is not dense in $L^p_{\rho}(\R^n)$, $1< p<\infty
$. As the dual of $L^p_{\rho}(\R^n) $ can be identified with
$L^{p'}(\R^n, (\varphi_{i\rho}(x))^{p'\gm_{p'}}dx )$ (where
$1/p+1/p'=1$), there exists $g\in L^{p'}(\R^n,
(\varphi_{i\rho}(x))^{p'\gm_{p'}}dx ) $ such that
$$\int_{\R^n}g(y)\tau_x\varphi_{\la}(y)dy =0$$
for all $\la \in A_1 $ and $x \in \R^n $. Note that $L^{p'}(\R^n,
(\varphi_{i\rho}(x))^{p'\gm_{p'}}dx ) \subset L^1(\R^n)$ for
$1<p'<\infty $ and hence the map
$$\la \mapsto \int_{\R^n}g(y)\tau_x\varphi_{\la}(y)dy $$
is a continuous function on $\Lambda_p^0$. Moreover, by Morera and Fubini,
the map is holomorphic. Since $A_1$ is a nonempty open subset of
$\Lambda_p^0$ it follows that
$$\int_{\R^n}g(y)\tau_x\varphi_{\la}(y)dy=0 $$
for all $\la \in \Lambda_p$; in particular, for all $\la \in \R$,
$x \in \R^n $, we have
$g\ast\varphi_{\la}(x)=\int_{\R^n}g(y)\tau_{-x}\varphi_{-\la}(y)dy
=0 $. In view of Fourier inversion formula, we have
$$g(x)= \int_0^\infty g\ast \varphi_{\la}(x)\la^{n-1}d\la=0.$$
 When $ p = 1 $ we have a bounded function $ g_1 $ such that
$$\int_{\R^n}g_1(y)\tau_x\varphi_{\la}(y) \tilde{K}_{n/2}(y) dy=0 $$
for all $\la \in \Lambda_p.$  Since the function $ g(y) = g_1(y)
\tilde{K}_{n/2}(y) $ belongs to $ L^1(\R^n) $ we can conclude that
$ g_1 = 0 $ as before. This proves the density of span of $
\mathcal{A}_1$. The density of span of $\mathcal{A}_2$ and
$\mathcal{A}_3$ are similarly proved.
\end{proof}

\noindent
 {\bf Proof of the Theorem \ref{T4}:}
For $1\leq p < \infty$ and $c\in \R$
 the semigroup $T_t^c=e^{-t(\D+\rho^2-c)}$ is strongly continuous on $L^p_{\rho}(\R^n).$ Therefore,
 in view of the Corollary \ref{C1} and Proposition \ref{P1}, $T_t^c$ is chaotic on
 $L^p_{\rho}(\R^n)$ for $c>c_p$ and $1\leq p<\infty,~p\neq 2 $.
 This proves the part (1) of Theorem \ref{T4}. Now we proceed to prove
 part (2), We note that for any $c\in \R$, the semigroup $T_t^c=e^{-t(\D+\rho^2-c)}$ is
 not hypercyclic on $L^{\infty}_{\rho}(\R^n)$, since
 for any $f \in L^{\infty}_{\rho}(\R^n)$, $T_t^c f $ is a continuous bounded
 function and hence the closure of the orbit $\{ T_t^cf: t>0 \} $ in
 $L^{\infty}_{\rho}(\R^n)$ is a subset of the subspace of all
 continuous bounded functions which is a strictly contained in $L^{\infty}_{\rho}(\R^n)$.
 This proves part (2) of Theorem \ref{T4}.
 For  part (3), we make use  of Theorem \ref{T8} according to which  the spectrum $ \si_p(\D+\rho^2-c) $ of the operator
 $(\D+\rho^2-c) $ on $L^p_{\rho}(\R^n) $ is given by
 $$ \mathfrak{P}_p -c =\{ \la^2+\rho^2-c: | \Im(\la)|\leq \gm_p \rho \}$$  for $1\leq p < \infty .$
 By the geometric form of the above set, it can be easily seen that
 the set $\si_p(\D+\rho^2-c)\cap i\R  $ has at most one point for $c\leq c_p$ and
 hence in view of Theorem \ref{T1}, $T_t^c$ is not chaotic.
 If $c<c_p$, the operators $T_t^c$ are uniformly bounded in $t$, as they satisfy the  estimates
 \begin{eqnarray*}
   \|T_t^cf(x)\|_{L^p_{\rho}(\R^n)}\leq C (1+t)^{\frac{n-1}{2}(1+\gm_p)} e^{-t(c_p-c)}
 \|f(x)\|_{L^p_{\rho}(\R^n)}.
 \end{eqnarray*}
for all $ f \in L^p_\rho(\R^n)$ and $1 \leq p < \infty .$
Consequently, for each $ f \in L^p_\rho(\R^n)$ and $c<c_p$ the
orbit $\{T_t^cf: t>0 \}$ is a bounded subset of $ L^p_\rho(\R^n)$
and hence it cannot be dense in $ L^p_\rho(\R^n)$. This proves that
$T_t^c$ is not hypercyclic on $ L^p_\rho(\R^n)$ for $c<c_p$ and $1
\leq p < \infty $ which completes the Theorem \ref{T4}.\\

\noindent
 {\bf Proof of the Theorem \ref{T3}:}   Parts (a) and (b) of Theorem \ref{T3} are
 restatements of Theorem \ref{T4}. Indeed, as we have already noted the chaotic
 behaviour of $ T_t^c $ on $L^p_\rho(\R^n) $ is the same as that of
 $ \tilde{T}_t^c $ on $  L^p(\R^n, (\varphi_{i\rho}(x))^2 dx) $
 as long as $ 2 < p < \infty.$  Thus Theorem \ref{T7} leads to the estimate
$$  \|\tilde{T}_tf(x)\|_{L^p(\R^n,(\varphi_{i\rho}(x))^2 dx)}$$
$$ \leq C (1+t)^{\frac{n-1}{2}(1+\gm_p)} e^{-\frac{4\rho^2}{p p'}t}
 \|f(x)\|_{L^p(\R^n, (\varphi_{i\rho}(x))^2 dx)}.$$
Similarly, from Theorem \ref{T8} it follows that the spectrum of $
\Delta+\rho^2 $ on $ L^p_\rho(\R^n) $ is the same as the spectrum
of $ \tilde{\Delta}+\rho^2 $ on $ L^p(\R^n,(\varphi_{i\rho}(x))^2
dx )$ for all $ p > 2.$  Thus the proof of parts (a) and (b) of
Theorem \ref{T3} is almost the same as that of Theorem \ref{T4}.

In order to treat the case $ 1 \leq p \leq 2 $ we make
use of the following theorem proved in \cite{DS} (see Theorem 2.5 in \cite{PS}).

\begin{thm}\label{T12}
Let $T_t$ be a hypercyclic semigroup generated by $A $ in a Banach
space $ B.$ Then the adjoint  $A^*$  of $A$  and the dual
semigroup $T_t^*$ on the dual space $B^*$  have the following
properties: (a) The point spectrum $ \sigma_{pt}(A^*)$  of $A^*$
is empty. (b) For any nonzero $ \phi \in B^*$, the orbit $
\{T_t^*\phi: t> 0\} $ is unbounded.
\end{thm}

Recalling the definition of $ \tilde{\Delta} $ we see that
$$ \int_{\R^n} \tilde{\Delta}f(x) g(x) (\varphi_{i\rho}(x))^2 dx =
\int_{\R^n} f(x)  \tilde{\Delta}g(x)  (\varphi_{i\rho}(x))^2 dx $$
which means that $ \tilde{\Delta}+\rho^2 $ is selfadjoint. If $
\tilde{T}_t^c $ were chaotic on $ L^p(\R^n,(\varphi_{i\rho}(x))^2
dx) $  for $ 1 \leq p < 2$, then the point spectrum of $
\tilde{\Delta}+\rho^2 $ on $ L^{p'} $ should be empty. But this is
not the case as $ \frac{\varphi_\lambda} {\varphi_{i\rho}} $ with
$ |\Im(\lambda)| < p' \gamma_{p'} $
belongs to $ L^{p'}(\R^n,(\varphi_{i\rho}(x))^2 dx) $ is an
eigenfunction of $ \tilde{\Delta}+\rho^2.$  For $ p = 2,$ the
behaviour of $ e^{-t(\tilde{\D}+\rho^2-c)} $ on $ L^2(\R^n, (\varphi_{i\rho}(x))^2 dx)$
is equivalent to the behaviour of $ e^{-t(\Delta+\rho^2-c)} $ on $ L^2(\R^n,dx).$
In \cite{PS} the authors have studied the latter semigroup and hence our results follow from theirs.\\

In order to show that $ \tilde{T}_t^c $ has no periodic points in
$L^{p}(\R^n,(\varphi_{i\rho}(x))^2 dx) $ for $ 1 \leq p \leq 2,$
assume, on the contrary that there is a nontrivial  $ f \in
L^{p}(\R^n,(\varphi_{i\rho}(x))^2 dx) $ such that $ \tilde{T}_t^c
f = f $ for some $ t = t_0 > 0.$ This means that $ g = f \varphi_{i\rho}
$ which belongs to $ L^{p}(\R^n,(\varphi_{i\rho}(x))^{2-p} dx) $
is a periodic point for  $ e^{-t(\Delta+\rho^2 -c)}$: that is,
$  e^{-t_0(\Delta+\rho^2 -c)}g = g.$ Since we are in the
case $ 1 \leq p \leq 2,   L^{p}(\R^n,(\varphi_{i\rho}(x))^{2-p}
dx) \subset  L^{1}(\R^n, dx) $ and hence by taking Fourier
transform we obtain $ (1-e^{-t_0\omega_c(\lambda)}) \hat{g}(\lambda
\omega) = 0 $ for all $ \lambda > 0 $ and $ \omega \in S^{n-1}.$
But then $ \hat{g}(\xi) = 0 $ for a.e. $ \xi \in \R^n $ which is a
contradiction. This completes the proof of Theorem \ref{T3}.

%--------------------------------------------------------------------------

\section{Chaotic behavior of the Dunkl heat semigroup\\
on weighted $L^p$ spaces}

%-----------------------------------------------------------------------------------

 \subsection{Coxeter groups and Dunkl operators:} In this subsection
 we recall some definitions given in Introduction and we give some more
 preliminaries about Dunkl theory.  Let $G$ be a Coxeter group
 (finite reflection group) associated to
  a fixed root system $R$ in $\R^n, n \geq 2$. We use the notation $\La . , . \Ra $
 for the standard inner product on $\R^n $ and $|x|^2 = \La x , x \Ra$. We assume that the
 reader is familiar with the notion of finite reflection groups
 associated to root systems. Given a root system $R$  we define
 the reflections $ \si_{\nu}$, $\nu \in R $ by
 $$\si_{\nu}x= x - 2 \; \frac{\La \nu, x \Ra}{|\nu |^2} \nu .$$
 Then $G$ is a subgroup of the orthogonal group generated by the reflections
 $\si_{\nu} $, $\nu \in R $. A function $\K $ defined on $R$ is called a
 multiplicity function if it satisfies $\K(g\nu)=\K(\nu)$ for every $g\in G$.
 We assume that our multiplicity function $\K $ is non negative.\\

 In \cite{DU} Dunkl defined a family of first order differential-difference
 operators $T_j$ (which we call Dunkl operators) that play the role of
 partial differentiation for the reflection group structure.
 Dunkl operators $T_j$ are defined by
 $$T_jf(x) = \frac{\pa }{\pa x_j}f(x) +\sum_{\nu \in R_+}\K(\nu)\nu _j
 \frac{f(x)-f( \si_{\nu}x)}{\La \nu, x \Ra}$$
 for $j=1,2,\cdots ,n $, where $\nu = (\nu_1, \nu_2, \cdots, \nu_n) $ and $R_+$ is the
 set of all positive roots in $R$. These operators map $\mathcal{P}^n_{m}$ to $\mathcal{P}^n_{m-1}$,
 where $\mathcal{P}^n_{m}$ is the space of homogeneous polynomials of degree $m$
 in $n$ variables. More importantly, these operators mutually
 commute; that is $T_iT_j=T_jT_i $.\\

   Recall that the Dunkl-Laplacian $\Dk $ is defined to be the operator
 $$ \Dk = -\sum _{j=1} ^d T_j ^2$$
 which can be explicitly calculated, see Theorem 4.4.9 in Dunkl-Xu
 \cite{DUX}. The Dunkl Laplacian reduces to the standard Laplacian $\Dk  = \D $ when $\K
 =0$. For all these facts we refer to Dunkl
 \cite{DU} and Dunkl-Xu \cite{DUX}. The  weight function $h_{\K}^2$
 associated to the group $ G $ and the multiplicity function  $\K$
 is defined by
 $$h^2_{\K}(x)=\prod_{\nu \in R_+}|\La x, \nu \Ra |^{2\K(\nu)}, \; \; x \in \R^n. $$
 Note that $h^2_{\K}(x) $ is a positive homogeneous function of degree $2\gm $ where
 $\gm = \sum_{\nu \in R_+}\K(\nu)$.  We consider $L^p $ spaces defined with
 respect to the measure $h^2_{\K}(x)dx $.
 There exists a kernel $E_{\K}(x,\; \xi) $
 which is a joint eigenfunction for all $T_j $:
 $$T_jE_{\K}(x,\; \xi)=\xi _j E_{\K}(x,\; \xi). $$
 This is the analogue of the exponential $e ^{\La x,\; \xi \Ra} $ and Dunkl transform
 is defined in terms of $E_{\K}(ix,\; \xi) $.\\

For $ f \in L^1(\R^n, h_\K(x)^2 dx) $ we define the Dunkl
transform of $ f $ by
$$ \mathcal{F}_\K f(\xi) = \int_{\R^n} f(x) E_\K(-ix, \xi) h_\K(x)^2 dx .$$
The Dunkl transform shares many important properties with the Fourier transform.
For example, we have the Plancherel theorem
$$\int_{\R^n} |\mathcal{F}_\K f(\xi) |^2 h_\K(\xi)^2 d\xi = c_n
\int_{\R^n} |f(x)|^2  h_\K(x)^2 dx $$ for all $ f \in L^1 \cap
L^2(\R^n, h_\K(x)^2 dx) $ and the inversion formula
$$ f(x) = c_n  \int_{\R^n} \mathcal{F}_\K f(\xi) E_\K(ix,\xi)  h_\K(\xi)^2 d\xi $$
for all $ f \in L^1(\R^n, h_\K(x)^2 dx) $ provided $ \mathcal{F}_\K f $ is also
in $L^1(\R^n, h_\K(x)^2 dx).$ In this paper we also make use of
some properties of the Dunkl kernel $ E_\K(x,\xi).$ For example we
require $ E_\K(\lambda x,\xi) =E_\K(x,\lambda \xi) $ for any $
\lambda \in \C $ and also the estimate $ |E_\K(x,\xi)| \leq
e^{|x||\xi|}$ for all $ x, \xi \in \R^n.$ We refer to \cite{DUX}
for all these and more on Dunkl transform.

 %Some properties of
% this kernel are listed below (see \cite{TX}, \cite{RM}).
% \begin{prop}\label{P2}
% For $x,y \in \R^n$, $\la \in \C $ and $g \in G $,
%    (a) $E_{\K}(x,y)=E_{\K}(y,x)$.
%    (b) $E_{\K}(\la x, y)=E_{\K}(x, \la y)$ and $E_{\K}(g x, g y)=E_{\K}(x,
%    y)$.
%    (c) $\overline{E_{\K}( x, y)}=E_{\K}(\bar{x},  \bar{y})$.
%    (d) $|E_{\K}( x, y)|\leq e^{|x|\:|y|}$.
% \end{prop}
% The function $E_{\K}(x, iy)$ plays the role of $e^{i \La x,\; y \Ra} $ in ordinary Fourier
% analysis. The Dunkl transform of a function $f \in L^1(\R^n, h_\K^2(x)dx)$ is defined in terms of it by
%\begin{equation}\label{I}
% \widehat{f}(y)=c_{\K}^{-1}\int_{\R^n}f(x)E_{\K}(x, -iy)h_{\K}^2(x)dx
%\end{equation}
%where $c_{\K}=\int_{\R^n}e^{-|x|^2/2}h_{\K}^2(x)dx $. If $\K=0 $
%then the Dunkl transform coincides with the usual Fourier
%transform. Some of the known properties of Dunkl transform are
%given below (see \cite{DU1}, \cite{DJ}).
%\begin{prop}\label{P3}
%For $f  \in L^1(\R^n, h_\K^2(x)dx) $, (1) $\widehat{f}$ is a
%continuous function vanishing at infinity. If $\widehat{f}$ is
%also in $L^1(\R^n, h_\K^2(x)dx)$ then we have the inversion
%formula
%$$f(x) = \int_{\R^n}\widehat{f}(y)E_{\K}(ix, y)h_{\K}^2(x) dx $$
%(2) The Dunkl transform extends to an isometry of $L^2(\R^n,
%h_\K^2(x)dx) $. (3) For Schwartz class function $f$,
%$\widehat{T_jf}(y)= iy_j\widehat{f}(y)$.
%\end{prop}

%--------------------------------------------------------------------------

\subsection{Dunkl heat semigroup on weighted $ L^p$ spaces}

  In \cite{RM} and \cite{RM1}, R\"{o}sler has studied the heat equation associated to the Dunkl
 Laplacian, viz.
 $$\frac{\pa}{\pa t}u(x,t)=\Dk u(x,t),\;\;u(x,0)=f(x),\; t>0, x \in \R^n .$$
 The solution of this equation, for $f \in L^p(\R^n, h_{\K}^2dx)$ is given by
 $$u(x,t)=\int_{\R^n}\Gamma_{\K}(t,x,y)f(y)h_{\K}^2(y)dy $$
 where $\Gamma_{\K}$ is the heat kernel associated to $\Dk$. The kernel $\Gamma_{\K}$ is
 explicitly known and is given by
 \begin{equation}\label{B}
   \Gamma_{\K}(t,x,y)=\frac{M_{\K}}{t^{n/2+\gm}}e^{-\frac{1}{4t}(|x|^2+|y|^2)}
   E_{\K}\Big (\frac{x}{\sqrt{2t}}, \frac{y}{\sqrt{2t}}\Big ).
\end{equation}
 We collect some important properties of this kernel in the
 following lemma.
 \begin{lem}[R\"{o}sler]\label{L3}$ $
 \begin{enumerate}
    \item $\Gamma_{\K}(t,x,y)=c_{\K}^{-2}\int_{\R^n}e^{-t|\xi|^2}
    E_{\K}(ix, \xi)E_{\K}(-iy, \xi)h^2_{\K}(\xi)d\xi$\\

    \item $\int_{\R^n}\Gamma_{\K}(t,x,y)h_{\K}^2(y)dy=1$\\

    \item  $\Gamma_{\K}(t+s, x, y)=\int_{\R^n}\Gamma_{\K}
    (t, x, z)\Gamma_{\K}(s, y, z)h_{\K}^2(z)dz $.
 \end{enumerate}
 \end{lem}
 In view of these properties, it is not difficult to show that the
 family of operators $ H_t$, $t>0$ defined on $L^p(\R^n, h_{\K}^2dx)$ by
 $$H_tf(x)=\int_{\R^n}\Gamma_{\K}(t,x,y)f(y)h_{\K}^2(y)dy$$
 forms a strongly continuous semigroup on $L^p(\R^n, h_{\K}^2 dx)$,
 $1\leq p<\infty$. Indeed, this has been proved in \cite{TX}.
 Thus for $f\in L^p(\R^n, h_{\K}^2 dx)$, $1\leq p<\infty$,
 $H_tf$ converges to $f$ in the norm as $t\rightarrow 0$.
 In this article we are interested in the semigroup $T_t$
 generated by $\Dk+\rho^2$ on the spaces $L^p_{\rho,\K}(\R^n)$. Note that
 $T_tf(x)=e^{-t\rho^2}H_tf(x) $ is well defined  even when
 $f\in L^p_{\rho,\K}(\R^n, dx)$
 as the integral defining $H_tf$ converges.\\

 We define spherical functions in the Dunkl set up by
 the equation
 $$\varphi_{\la,\K}(x)=\int_{\sn}E_{\K}(ix, \la \om)h_{\K}^2(\om)d\si(\om)$$
  for $\la\in\C$. These are all eigenfunctions of the Dunkl-Laplacian with eigenvalue $\la^2$;
 $\D_{\K}\varphi_{\la,\K}=\la^2\varphi_{\la,\K}$. For $\la \in \C $, $\varphi_{\la,\K} $ has exponential growth.
 Indeed,
 $$\varphi_{\la,\K}(x)=c_n\frac{J_{\frac{n}{2}+\gm-1}(\la|x|)}{(\la|x|)^{\frac{n}{2}+\gm-1}}$$
 where $J_{\ap}(t)$ is the Bessel function of type $\ap$. It can
 be easily proved that $\varphi_{\la,\K} \in L^p_{\rho,\K}(\R^n)$
 for $|\Im(\la)|< \gm_p \rho $ and $p\neq 2 $.\\

 \begin{thm}\label{T13}
 For each $1\leq p<\infty$, $T_t$, $t>0$ defines a strongly continuous semigroup on
 $L^p_{\rho,\K}(\R^n)$. Moreover, for any $f \in L^p_{\rho,\K}(\R^n)$,
  $1\leq p \leq 2,$
  $$\|T_tf\|_{L^p_{\rho,\K}(\R^n)} \leq C t^{\frac{(n+2\gm-1)}{2}\gm_p}
 e^{-\frac{2\rho^2}{p'}t}\|f\|_{L^p_{\rho,\K}(\R^n)}$$
 whereas for $p>2$
 $$\|T_tf\|_{L^p_{\rho,\K}(\R^n)} \leq C t^{\frac{(n+2\gm-1)}{2}\gm_p}
 e^{-\frac{2\rho^2}{p}t}\|f\|_{L^p_{\rho,\K}(\R^n)}.$$
 \end{thm}
 \begin{proof}First note that the dual space of
 $L^p_{\rho,\K}(\R^n)$ can be identified with $L^{p'}(\R^n,
 (\varphi_{i\rho,\K}(x))^{p'\gm_{p'}}h_{\K}^2(x)dx)$ and with this
 identification the operator $T_t$ will be self adjoint.
 In view of the asymptotic behaviour of the Macdonald function (and Bessel function),
it is enough to consider the space defined using $
(1+|x|)^{(n+2\gm-1)/2}e^{-\rho |x|}$ (respectively
$(1+|x|)^{-(n+2\gm-1)/2}e^{\rho |x|} $ ) in place of $
\tilde{K}_{n/2+\gm}(\rho |x|)$ (respectively
$(\varphi_{i\rho,\K}(x))^{p'\gm_{p'}}$). For the sake of brevity,
just for this section, we denote the weight function $
(1+|x|)^{(n+2\gm-1)/2}e^{-\rho |x|}$ by $w_{\rho,\K}(x).$\\

 We first consider the case $1\leq p\leq 2$ for which we prove the above
 estimates for the semigroup $T_t$ on both weighted $L^p$-spaces $L^p_{\rho,\K}(\R^n)$ and $L^{p}(\R^n,
 (\varphi_{i\rho,\K}(x))^{p\gm_{p}}h_{\K}^2(x)dx)$. We make use of these estimates
 and duality to prove the required estimates for $p>2$. Since
 $$T_tf(x)=e^{-t\rho^2}\int_{\R^n}\mathcal{F}_\K f(\xi)e^{-t|\xi|^2}E_{\K}(ix, \xi)h_{\K}^2(\xi)d\xi$$
 it is clear that
 $$\|T_tf\|_{2}\leq C e^{-t\rho^2}\|f\|_2, \; \; f \in L^2(\R^n, h_\K^2(x)dx ).$$
 Since $L^2_{\rho,\K}(\R^n)=L^2(\R^n, h_\K^2(x)dx)=(L^2_{\rho,\K}(\R^n))^*$ the
 required estimate is true for $p=2.$ To prove the result for $p=1$
 we recall that
 $$T_tf(x)=e^{-t\rho^2}\int_{\R^n} \Gamma_{\K}(t,x,y)f(y)h_{\K}^2(y)dy$$
 where $\Gamma_{\K}(t,x,y)$ is the heat kernel defined in \eqref{B}. As $\gm_1=1$ we need to
 show that
 $$\sup_{y\in \R^n}(w_{\rho,\K}(y))^{\mp 1}\int_{\R^n}\Gamma_{\K}(t,x,y)(w_{\rho,\K}(x))^{\pm 1}
 h_{\K}^2(x)dx\leq Ct^{(n+2\gm-1)}e^{t\rho^2}.$$
 We consider the case of $L^1_{\rho,\K}(\R^n)$; the treatment of $L^{1}(\R^n,
 \varphi_{i\rho,\K}(x)h_{\K}^2(x)dx)$ is similar.\\

 Since the heat kernel $\Gamma_{\K}(t,x,y)$ satisfies
 $\int_{\R^n}\Gm_{\K}(t,x,y)h_{\K}^2(x)dx=1 $
 we immediately see that
 $$(1+|y|)^{-(n+2\gm-1)/2}e^{\rho|y|}\int_{|x|\geq |y|}\Gm_{\K}(t,x,y)(1+|x|)^{(n+2\gm-1)/2}e^{-\rho|x|}h_{\K}^2(x)dx \leq C.$$
 In order to treat the remaining part of the integral, we make use
 of the explicit expression for $\Gm_{\K}(t,x,y)$, viz.
 $$\Gm_{\K}(t,x,y)=M_{\K}t^{-N/2}e^{-\frac{1}{4t}(|x|^2+|y|^2)}
 E_{\K}\Big (\frac{x}{\sqrt{2t}}, \frac{y}{\sqrt{2t}}\Big ) $$
 where we have written $N=n+2\gm$. Integrating in polar coordinates and
 making use of the formula (see Proposition 2.3 in \cite{TX})
 \begin{equation}\label{C}
   \int_{\sn}E_{\K}(rx', sy')h_{\K}^2(x')d\si(x')=
 c_N\frac{J_{\frac{N}{2}-1}(irs)}{(irs)^{\frac{N}{2}-1}}
\end{equation}
 we need to estimate
 $$t^{-\frac{N}{2}}(1+s)^{-\frac{N-1}{2}}e^{\rho s}\int_0^{s}e^{-\frac{1}{4t}(r^2+s^2)}
 \frac{J_{\frac{N}{2}-1}(i\frac{rs}{2t})}
 {(i\frac{rs}{2t})^{\frac{N}{2}-1}}(1+r)^{\frac{N-1}{2}}e^{-\rho r}r^{N-1}dr.$$
 In view of the Poisson integral representation of Bessel
 functions, viz.
 $$J_{\ap}(t)= \Big (\frac{t}{2}\Big )^{\ap}\int_{-1}^1e^{iut}(1-u^2)^{\ap-\frac{1}{2}}du $$
 we need to estimate the integral
 $$t^{-\frac{N}{2}}(1+s)^{-\frac{N-1}{2}}e^{\rho s}\int_0^{s}e^{-\frac{1}{4t}(r^2+s^2-2rsu)}
 (1+r)^{\frac{N-1}{2}}e^{-\rho r}r^{N-1}dr.$$
 Note that $(s-r)^2 = r^2+s^2-2rs\leq r^2+s^2-2rsu $ for any $-1\leq u\leq 1$.
 Consequently, as $s\geq r$,
 $(s-r)\leq (r^2+s^2-2rsu)^{\frac{1}{2}} $ and $ (1+s) \geq (1+r),$
  the above integral is bounded by
 $$t^{-\frac{N}{2}}\int_0^{s}e^{-\frac{1}{4t}(r^2+s^2-2rsu)}
 e^{\rho(r^2+s^2-2rsu)^{\frac{1}{2}}}r^{N-1}dr$$
 $$\leq t^{-\frac{N}{2}}\int_0^{\infty}e^{-\frac{1}{4t}(r^2+s^2-2rsu)}
 e^{\rho(r^2+s^2-2rsu)^{\frac{1}{2}}}r^{N-1}dr. $$
 Thus, the required integral is bounded by
 $$t^{-\frac{N}{2}}\int_0^{\infty}\int_{-1}^1 (1-u^2)^{\frac{N-1}{2}}
 e^{-\frac{1}{4t}(r^2+s^2-2rsu)}e^{\rho(r^2+s^2-2rsu)^{\frac{1}{2}}}r^{N-1}du dr.$$
 The inner integral is the generalised Euclidean translation of
 $e^{-\frac{1}{4t}r^2}e^{\rho r}.$  As the $L^1$-norm is preserved
 by such a translation, the above is bounded by
 $$t^{-\frac{N}{2}}\int_0^{\infty}e^{-\frac{1}{4t}r^2}e^{\rho r} r^{N-1}dr $$
 which can be easily seen to be bounded by $ C  e^{t\rho^2} t^{\frac{N-1}{2}} $.\\

 We can now appeal to Stein-Weiss interpolation theorem (see in \cite{SW} )
 to prove the result for $1\leq p\leq 2$. Indeed, we have
 $$\int_{\R^n}|T_t f(x)|(w_{\rho,\K}(x))^{\pm 1}h_{\K}^2(x)dx \leq C t^{\frac{N-1}{2}}\int_{\R^n}|f(x)|(w_{\rho,\K}(x))^{\pm 1}h_{\K}^2(x)dx $$
 and also
 $$\left(\int_{\R^n}|T_t f(x)|^2h_{\K}^2(x)dx\right)^{\frac{1}{2}} \leq C e^{-t\rho^2} \left(\int_{\R^n}|T_t f(x)|^2h_{\K}^2(x)dx\right)^{\frac{1}{2}}.$$
 Interpolation of these two estimates give us
 $$\left( \int_{\R^n}|T_tf(x)|^p (w_{\rho,\K}(x))^{\pm p \gm_p }h_{\K}^2(x)dx \right)^{\frac{1}{p}}$$
 $$\leq C t^{\frac{N-1}{2}\gm_p}e^{-\frac{2\rho^2}{p'}t}
 \left(\int_{\R^n}|f(x)|^p(w_{\rho,\K}(x))^{\pm p \gm_p \rho |x|}h_{\K}^2(x)dx \right)^{\frac{1}{p}} $$
 which is the required inequality for $1\leq p\leq 2$. In order to prove
 Theorem \ref{T13} when $p>2 $ we use duality. Observe that
 $$\int_{\R^n}T_tf(x)g(x)h_{\K}^2(x)dx=\int_{\R^n}f(x)T_tg(x)h_{\K}^2(x)dx  .$$
 Writing the right hand side as
 $$\int_{\R^n}f(x)(w_{\rho,\K}(x))^{ \gm_p }T_tg(x)(w_{\rho,\K}(x))^{ -\gm_p }h_{\K}^2(x)dx $$
 and applying Holder's inequality we get
 \begin{eqnarray*}
  \left|\int_{\R^n}T_tf(x)g(x)h_{\K}^2(x)dx\right| &\leq&  C \left(\int_{\R^n}|f(x)|^p(w_{\rho,\K}(x))^{p \gm_p}h_{\K}^2(x)dx \right)^{\frac{1}{p}} \\
    &\times& \left(\int_{\R^n}|T_tg(x)|^{p'}(w_{\rho,\K}(x))^{ -p' \gm_p }h_{\K}^2(x)dx \right)^{\frac{1}{p'}}  .
 \end{eqnarray*}
 For $p>2 $, $p'<2$ and hence by what we have already proved and
 the fact that $\gm_p =\gm_{p'}$, we get
 $$|\int_{\R^n}T_tf(x)g(x)h_{\K}^2(x)dx|$$
 $$\leq C t^{\frac{N-1}{2}\gm_p}e^{-\frac{2\rho^2}{p}t}\|f\|_{L^p_{\rho,\K}(\R^n)}
 \|g\|_{(L^{p}_{\rho,\K}(\R^n))^*}.$$
 Taking supremum over all $g \in (L^{p}_{\rho,\K}(\R^n))^* $ we obtain the required
 estimate.\\

 The strong continuity of $T_t$ on $L^p_{\rho,\K}(\R^n)$ follows from the norm estimates. Indeed,
 for $0<t\leq 1$, the operators $T_t$ are uniformly bounded on
 $L^p_{\rho,\K}(\R^n)$, $1\leq p <\infty$. As $L^p(\R^n, h_{\K}^2(x)dx)$
 is dense in $L^p_{\rho,\K}(\R^n)$, the strong continuity of $T_t$
 on $L^p_{\rho,\K}(\R^n)$ follows from the same on $L^p(\R^n, h_{\K}^2(x)dx)$ in view of
 the continuous inclusion $L^p(\R^n, h_{\K}^2(x)dx)\subset L^p_{\rho,\K}(\R^n)$.
 The proof of strong continuity of $T_t$ on $L^p(\R^n,
 h_{\K}^2(x)dx)$ was given in Theorem 5.3 of \cite{TX}.
 \end{proof}

We now turn our attention to the Dunkl heat semigroup on the
weighted mixed norm space
 $L^{p,2}_{\rho, \K}(\R^n)$. The semigroup $T_t$ can be
 extended to the space $L^{p,2}_{\rho,\K}(\R^n).$
 In fact we will show below that the weighted mixed norm ($L^{p,2}_{\rho,\K}(\R^n)$-norm)
 estimate of $T_tf$ can be reduced to a  vector valued inequality
 for  a sequence of Bessel semigroups of different types.\\

 The Bessel
 semigroup $B_t^{\ap}$ of type $\ap$ is initially defined on $L^2(\R^+, r^{2\ap+1}dr)$ by
 \begin{equation}\label{D}
    B_t^{\ap}f(r)= \int_0^{\infty}f(s)b_t^{\ap}(r,s)s^{2\ap+1}ds
\end{equation}
 where the kernel $b_t^{\ap}(r,s)$ is given by
 \begin{equation}\label{E}
   b_t^{\ap}(r,s)=(2t)^{-1}e^{-\frac{1}{4t}(r^2+s^2)}
 (rs)^{-\ap}J_{\ap}(\frac{irs}{2t})
\end{equation}
  where $J_{\ap}$ is the standard Bessel function of type $\ap$ of first
  kind.\\

      We can identify $L^{p,2}_{\rho,\K}(\R^n)$
  with $L^p(\R^+, \Hc, (\tilde{K}_{n/2+\gamma}(\rho r))^
  {p \gamma_p}r^{n+2\gm-1}dr),$ the $L^p$ space of $\Hc$ valued functions
  defined on $\R^+$ taken with respect to the measure
  $(\tilde{K}_{n/2+\gamma}(\rho r))^{p \gamma_p}r^{n+2\gm-1}dr$ where
  $\Hc = L^2(\sn, h_{\K}^2(\om)d\si(\om))$. For the space
  $L^2(\sn, h^2_{\K}(\om)d\si(\om))$ there exists an orthonormal
  basis consisting of h-harmonics. These are analogues of spherical
  harmonics and defined using $\D_{\K}$ in place $ \D $.
  A homogeneous polynomial $P(x)$ is said to be a solid
  h-harmonic if $\D_{\K} P(x)=0$. Restrictions of such solid harmonics to $\sn$ are
  called spherical h-harmonics. The space $L^2(\sn, h^2_{\K}(\om)d\si(\om))$ is the orthogonal
  direct sum of the finite dimensional spaces $\Hc_{m}^h $ consisting of
  h-harmonics of degree $m $. We can choose an orthonormal basis
  $Y_{m,j}^h $, $j=1,2,\ldots ,d(m) $, $d(m)=dim (\Hc_{m}^h) $ for $\Hc_{m}^h $ so that the collection
  $\{Y_{m,j}^h: \; j=1,2,\ldots ,d(m),\; m = 0,1,2,\ldots  \} $ is an orthonormal basis
  for $L^2(\sn, h^2_{\K}(\om)d\si(\om))$.\\

  If $f \in L^{p,2}_{\rho, \K}(\R^n),$ then $f(r\; \cdot) \in L^2(\sn,
  h^2_{\K}(\om)d\si(\om))$ for almost every  $r.$ Hence we
  have the following h-harmonic expansion:  for a.e. $r>0$,
  $$f(r\om)=\sum_{m=0}^{\infty}\sum_{j=1}^{d(m)}f_{m,j}(r)Y_{m,j}^h(\om)$$
  where $f_{m,j}(r)= \int_{\sn}f(r\om)Y_{m,j}^h(\om)h_{\K}^2(\om)d\si(\om)$
  are the spherical harmonic coefficients of $ f .$  In  view of Plancherel formula, we also have the following
expression for $  \|f\|_{L^{p,2}_{\rho,\K}(\R^n)} $:
  \begin{equation}\label{F}
    \left(\int_0 ^{\infty}
  \Big ( \sum_{m=0}^{\infty}\sum_{j=1}^{d(m)}|f_{m,j}(r)|^2
 \Big )^{\frac{p}{2}}(\tilde{K}_{n/2+\gamma}(\rho r))^{p \gamma_p}r^{n+2\gm-1}dr\right)
 ^{\frac{1}{p}}.
\end{equation}
  With the above notations, the following proposition gives the relation
  between the Dunkl heat semigroup and the Bessel semigroups.\\

  \begin{prop}\label{P4}
  For $1\leq p < \infty$, let $ T_t$ be the semigroup generated by $A=\Dk+\rho^2$
   and $f \in L^{p,2}_{\rho, \K}(\R^n)$. Then we have
  $$\int_{\sn}T_t f(r\om)Y_{m,j}^h(\om) h_{\K}^2(\om)d\si(\om)= c_{n,\K}
  e^{-t\rho^2}r^m B_t^{\frac{N}{2}+m-1}(\widetilde{f}_{m,j})(r)$$
  for $m=0, 1, 2, \ldots$ and $j=1, 2, \ldots , d(m)$ where $\widetilde{f}_{m,j}(r)=r^{-m}f_{m,j}(r). $
  \end{prop}
  \begin{proof}
   To prove the proposition, we make use of the following formula
\begin{eqnarray}\label{A}
  & &\int_{\sn}E_\K(x,y)Y_{m,j}^h(x')h_{\K}^2(x')d\si(x') \\  &=&
  c_{n,\K}(|x|\:|y|)^{-(\frac{N}{2}-1)}J_{\frac{N}{2}+m-1}(i |x|\:|y|)Y_{m,j}^h(y')\notag .
\end{eqnarray}
 In view of \eqref{B} and \eqref{A}, we have
 \begin{eqnarray*}
   & &\int_{\sn}\Gm_{\K}(t,r\om,s\eta)Y_{m,j}^h(\om)h_{\K}^2(\om)d\si(\om)\\ &=&
   \frac{M_{\K}}{t^{\frac{N}{2}}}e^{-t\rho^2} e^{-\frac{1}{4t}(r^2+s^2)}
   \int_{\sn}E_{\K}\Big (\frac{r\om}{\sqrt{2t}}, \frac{s\eta}{\sqrt{2t}}\Big )
   Y_{m,j}^h(\om)h_{\K}^2(\om)d\si(\om)\\\notag
   &=& c_{n,\K}e^{-t\rho^2} (2t)^{-1}e^{-\frac{1}{4t}(r^2+s^2)}
   (rs)^{-(\frac{N}{2}-1)}J_{\frac{N}{2}+m-1}\Big (\frac{i rs}{2t}\Big )Y_{m,j}^h(\eta).\notag
 \end{eqnarray*}
 Recalling the definition of $T_tf$ and making use of the above
 formula,
 \begin{eqnarray*}
  & & \int_{\sn}T_tf(r\om)Y_{m,j}^h(\om)h_{\K}^2(\om)d\si(\om)\\
  &=&  c_{n,\K}e^{-t\rho^2} (2t)^{-1}\int_{0}^{\infty}e^{-\frac{1}{4t}(r^2+s^2)}
   (rs)^{-(\frac{N}{2}-1)}J_{\frac{N}{2}+m-1}\Big (\frac{i rs}{2t}\Big
   )f_{m,j}(s)s^{N-1}ds\\
   &=& c_{n,\K} e^{-t\rho^2}r^m
   B_t^{\frac{N}{2}+m-1}(\widetilde{f}_{m,j})(r).
 \end{eqnarray*}
This proves the proposition.
 \end{proof}
 \begin{prop}\label{P5}
  For $1\leq p < \infty$, let $ T_t$ be the semigroup generated by $A=\Dk+\rho^2$
   and $f \in L^{p,2}_{\rho, \K}(\R^n)$. Then
  $$\|T_tf\|_{L^{p,2}_{\rho, \K}(\R^n)}\leq  A(p,t)\|f\|_{L^{p,2}_{\rho, \K}(\R^n)} $$
  if and only if the vector-valued inequality holds:
  $$\left(\int_0 ^{\infty}
  \Big ( \sum_{m=0}^{\infty}\sum_{j=1}^{d(m)}|r^m B_t^{\frac{N}{2}+m-1}(\widetilde{f}_{m,j})(r)|^2
 \Big )^{\frac{p}{2}}(\tilde{K}_{n/2+\gamma}(\rho r))^
  {p \gamma_p}r^{N-1}dr)\right)^{\frac{1}{p}}$$
 $$\leq  A(p,t) \left(\int_0 ^{\infty}
  \Big ( \sum_{m=0}^{\infty}\sum_{j=1}^{d(m)}|f_{m,j}(r)|^2
 \Big )^{\frac{p}{2}}(\tilde{K}_{n/2+\gamma}(\rho r))^
  {p \gamma_p}r^{N-1}dr) \right)
 ^{\frac{1}{p}}.$$
Here $N=n+2\gm $, $f_{m,j}(s)=s^{-m}f_{m,j}(s)$ and $A(p,t)$ is a constant depending on $p$ and $t$.
  \end{prop}
  \begin{proof}
With $F=T_tf $ we use the h-harmonic expansion to get
 $$\|T_tf\|_{L^{p,2}_{\rho, \K}(\R^n)}= \left(\int_0 ^{\infty}
  \Big ( \sum_{m=0}^{\infty}\sum_{j=1}^{d(m)}|F_{m,j}(r)|^2
 \Big )^{\frac{p}{2}}(\tilde{K}_{N/2}(\rho r))^
  {p \gamma_p}r^{N-1}dr)\right)
 ^{\frac{1}{p}}. $$
 Since $F_{m,j}(r) =
 \int_{\sn}T_tf(r\om)Y_{m,j}^h(\om)h_{\K}^2(\om)d\si(\om)$, in
 view of the previous proposition and the above, the norm $\|T_tf\|_{L^{p,2}_{\rho, \K}(\R^n)}
 $ is equal to
 \begin{eqnarray*}
 c_{n,\K} e^{-t\rho^2}\left(\int_0 ^{\infty}
  \Big ( \sum_{m=0}^{\infty}\sum_{j=1}^{d(m)}|r^m B_t^{\frac{N}{2}+m-1}(\widetilde{f}_{m,j})(r)|^2
 \Big )^{\frac{p}{2}}(\tilde{K}_{N/2}(\rho r))^
  {p \gamma_p}r^{N-1}dr\right)
 ^{\frac{1}{p}}
 \end{eqnarray*}
which proves the proposition.\\
\end{proof}
 In Theorem \ref{T13}, we have obtained a  bound for the operator norm of $T_t$
 on $L^p_{\rho,\K}(\R^n)$ which is given by $C~ t^{\frac{n+2\gm-1}{2}\gm_p}
 e^{-\frac{2\rho^2}{p'}t} $ or $~C~ t^{\frac{n+2\gm-1}{2}\gm_p}
 e^{-\frac{2\rho^2}{p}t} $ depending on whether $1\leq p \leq 2 $ or $p>2 .$
 This bound can be improved  if we consider
 the heat semigroup on $L^{p,2}_{\rho,\K}(\R^n)$ under the added assumption that $2\gm$ is an integer. \\
 \begin{thm}\label{T14}
 Let $1\leq p < \infty $ and let $2\gm $ be an integer. Then $T_t$, $t>0$
 defines a strongly continuous semigroup on $L^{p,2}_{\rho,\K}(\R^n)$.
 Moreover, for any $f \in L^{p,2}_{\rho,\K}(\R^n)$
 we have
 \begin{equation}\label{H}
   \|T_tf\|_{L^{p,2}_{\rho,\K}(\R^n)}\leq C (1+t)^{\frac{n+2\gm-1}{2}(1+\gm_p)}
 e^{-\frac{4\rho^2}{p p'}t} \|f\|_{L^{p,2}_{\rho,\K}(\R^n)} .
\end{equation}
 \end{thm}
 \begin{proof}
 Let $f(x)=f_0(r)Y(\om)$ where $x=r\om$, $\om \in \sn $, $r=|x| $
 and let $Y(\om)$ be  a spherical h-harmonic of degree $m$. In view of
 Proposition \ref{P4}, we have
 $$T_tf(r\om) =  c_{n,\K} e^{-t\rho^2}r^m B_t^{\frac{N}{2}+m-1}(\widetilde{f}_0)(r) Y(\om)$$
 where $\widetilde{f}_0(r) = r^{-m}f_0(r)$. Note that  $ \|T_tf -f\|_{L^p_{\rho,
 \K}(\R^n)}$ which is equal to the product of
 \begin{eqnarray*}
    & & \left(\int_0 ^{\infty}
   |c_{n,\K} e^{-t\rho^2} r^m  B_t^{\frac{N}{2}+m-1}(\widetilde{f}_0)(r)-
   f_0(r)|^p (\tilde{K}_{N/2}(\rho r))^
  {p \gamma_p}r^{N-1}dr\right)^{\frac{1}{p}}
  \end{eqnarray*}
with $   \Big ( \int_{\sn}|Y(\om)|^p h_{\K}^2(\om)d\si(\om)\Big
    )^{\frac{1}{p}},$
tends to 0 as $t\rightarrow 0$ by the strong continuity of
 $T_t$ on $L^p_{\rho, \K}(\R^n)$ for $1\leq p < \infty $.
 This implies $\|T_tf-f\|_{L^{p,2}_{\rho,\K}(\R^n)} \rightarrow
 0$ as $t\rightarrow 0$, as $Y \neq 0 $. Similarly, if $f \in L^{p,2}_{\rho,\K}(\R^n)$ is of
 the form
 \begin{equation}\label{G}
    f(r\om) = \sum_{m=0}^M\sum_{j=1}^{d(m)}f_{m,j}(r)Y_{m,j}(\om)
\end{equation}
 where $M$ is a positive integer, then it follows that $\|T_tf-f\|_{L^{p,2}_{\rho,\K}(\R^n)}
 \rightarrow 0$ as $t\rightarrow 0$. Since the space of all such functions $ f$ having  the
 form \eqref{G} is dense in $L^{p,2}_{\rho,\K}(\R^n),$ once we prove that $T_t $
are uniformly bounded on $L^{p,2}_{\rho,\K}(\R^n) $ for $0< t\leq 1$
it is immediate that $\|T_tf-f\|_{L^{p,2}_{\rho,\K}(\R^n)} \rightarrow
 0$ as $t\rightarrow 0$ for every $f \in L^{p,2}_{\rho,\K}(\R^n)$
 and hence $T_t$ is strongly continuous on
 $L^{p,2}_{\rho,\K}(\R^n)$.\\

 In view of Proposition \ref{P5}, in order to prove the weighted mixed norm estimate \eqref{H}, it is enough
 to prove the following vector-valued inequality
 $$\left(\int_0 ^{\infty}
  \Big ( \sum_{m=0}^{\infty}\sum_{j=1}^{d(m)}|r^m B_t^{\frac{N}{2}+m-1}(\widetilde{f}_{m,j})(r)|^2
 \Big )^{\frac{p}{2}}(\tilde{K}_{N/2}(\rho r))^
  {p \gamma_p}r^{N-1}dr\right)^{\frac{1}{p}}$$
 $$\leq A(p,t)
  \left(\int_0 ^{\infty}
  \Big ( \sum_{m=0}^{\infty}\sum_{j=1}^{d(m)}|f_{m,j}(r)|^2
 \Big )^{\frac{p}{2}}(\tilde{K}_{N/2}(\rho r))^
  {p \gamma_p}r^{N-1}dr \right)
 ^{\frac{1}{p}}$$
where $ A(p,t) = C
(1+t)^{\frac{n+2\gm-1}{2}(1+\gm_p)}e^{-\frac{4\rho^2}{p p'}t} .$
 In view of the same Proposition \ref{P5},  the above vector valued inequality will follow once we prove
 $$\|T_tf\|_{L^{p,2}_{\rho,0}(\R^N)}\leq C (1+t)^{\frac{n+2\gm-1}{2}(1+\gm_p)}
 e^{-\frac{4\rho^2}{p p'}t} \|f\|_{L^{p,2}_{\rho, 0}(\R^N)} $$
for the standard heat semigroup  $T_t= e^{-t(\D+\rho^2)}$ on
$\R^N, N = n+2\gamma$. This is
 the content of the next theorem.
 \end{proof}
 \begin{thm}\label{T15}
 Let $1\leq p <\infty $, $\D$ be the standard Laplacian on $\R^N$ and
 $T_t $ be the semigroup generated by $(\D+\rho^2)$.
 Then
$$\|T_tf\|_{L^{p,2}_{\rho, 0}(\R^N)}\leq C
(1+t)^{\frac{n+2\gm-1}{2}(1+\gm_p)}
 e^{-\frac{4\rho^2}{p p'}t} \|f\|_{L^{p,2}_{\rho, 0}(\R^N)}.$$
 \end{thm}
We obtain the above result as a consequence of the weighted norm
estimate proved in  Theorem \ref{T7}. To this end, we make use of
a transference result due to Rubio de Francia, see \cite{RF}.
 For given $k\in SO(N),$ the special orthogonal group, we
 define the rotation operator $\varrho(k)$  by
 $\varrho(k)f(x)=f(kx)$. For a given radial weight function $ w $ consider  weighted mixed norm space $
 L^{p,2}_{w}(\R^N)$ consisting  of all functions $f$
 for which the norms
 $$\|f\|_{L_{w}^{p,2}(\R^N)} = \left(\int_0^{\infty}
\left(\int_{\sN}|f(r\om)|^2h_{\K}^2(\om)d\si(\om) \right)
 ^{\frac{p}{2}}w(r)r^{N-1}dr \right)^{\frac{1}{p}}$$
 are finite. We claim that for any bounded linear operator $T$ acting on
$L^p(\R^N,w(|x|)dx)$ which commutes with
 rotations, i.e. $T\varrho(k)=\varrho(k)T$ for every
 $k\in SO(N)$, there exists a bounded linear operator $\tilde{T}$ on $L^{p,2}_{w}(\R^N)$ such that $\tilde{T}f=Tf $ for
 $f \in L^{p,2}_{w}(\R^N)\cap L^p(\R^N, w(|x|)dx)$
 and $\|\tilde{T}\|_{op}\leq \|T\|_{op} $.
 In order to prove this claim, we  make use of an idea due to  Rubio de Francia \cite{RF}.
 This method described briefly in \cite{RF} is based on an extension
 of a theorem of Marcinkiewicz and Zygmund as expounded in Herz
 and Riviere \cite{HR} in the form of the following
 lemma.

 \begin{lem}\label{L4}
 Let $( G, \mu )$ and $( H, \nu )$ be arbitrary measure spaces and $S:
 L^p(G)\rightarrow L^p(G)$ a bounded linear operator. Then if $p\leq q\leq 2 $
 or $p\geq q \geq 2$, there exists a bounded linear operator
 $\tilde{S}: L^p(G; L^q(H))\rightarrow L^p(G;L^q(H))
 $ with $\|  \tilde{S}  \| \leq\|  S  \| $ such that for
 $g\in L^p(G; L^q(H)) $ of the form $g(x,\xi)=f(\xi)u(x)$
 where $f\in  L^p(G)$ and $u\in L^q(H) $ we have
 $$(\tilde{S}g)(\xi, x)=(Sf)(\xi)u(x) .$$
 \end{lem}
 The idea of Rubio de Francia is as follows. Since $
 T: L^p(\R^N, w(|x|)dx) \rightarrow L^p(\R^N, w(|x|)dx) $ is a bounded
 linear operator, by the lemma of Herz and Riviere, there exists a bounded linear operator $
 \tilde{T} $ on $L^p(\R^N, \mathcal{H}, w(|x|)dx),$ the space of all $\mathcal{H}$
 valued functions $F$ on $\R^N$ for which
 $$\int_{\R^N} \Big ( \int _K |F(x)(k)|^2 dk\Big )^{\frac{p}{2}}w(|x|)dx$$
 are finite. Here
 $ \mathcal{H} $ is the Hilbert space $ L^2(K)$, $K=SO(N)$ and $dk$ is the Haar measure on $K$.
 Moreover, the operator $\tilde{T}$ satisfies $(\tilde{T}\tilde{f})(x,k) =
 Tg(x)h(k)$ if $\tilde{f}(x,k)=g(x)h(k) $, $x \in \R^N $, $k\in SO(N) $.
 Given a function $f \in L^p(\R^N, w(|x|)dx) $ consider $\tilde{f}(x, k) = \varrho(k)f(x)=f(kx)
 $. Then $$\int _{\R^N}\left(\int _K | \tilde{f}(x,k)|^2 dk \right)^{\frac{p}{2}}w(|x|)dx $$
 can be calculated as follows.
 If $x=r\om $ , $\om \in \sN $, $\tilde{f}(x,k)=f(rk\om) $ and
 hence
 \begin{eqnarray}
  \int _K |\tilde{f}(x,k)|^2dk =
  \int_{K_{\om}}\Big (\int_{K/K_{\om}}|f(r k \om )|^2 d \mu \Big )d\nu
 \end{eqnarray}
 where $K_{\om}=\{ k\in K: k \om = \om \} $ is the isotropy
 subgroup of $ K $, $ d\nu $ is the Haar
 measure on $K_{\om} $ and $d\mu $ is the $K_{\om} $ invariant
 measure on $K/K_{\om} $ which
 can be identified with $\sN $. Hence
 \begin{eqnarray*}
   \int _K |\tilde{f}(x,k)|^2 dk = c_N \int _{\sN} |f(r\om)|^2d\sigma
   (\om).
 \end{eqnarray*}
 Therefore,
 \begin{eqnarray*}
  & &\int_{\R^N} \Big ( \int _K |\tilde{f}(x,k)|^2 dk\Big
  )^{\frac{p}{2}}w(|x|)dx\\
  & = &  c'_N \int _0 ^{\infty} \Big ( \int _{\sn} |f(r\om)|^2d\sigma
   (\om)\Big )^{\frac{p}{2}}w(r)r^{N-1}dr.
 \end{eqnarray*}
 and hence $L^{p,2}_{w}(\R^N)$ can be considered as a subspace of $L^p(\R^N, \mathcal{H}, w(|x|)dx)$
  with the identification $f\mapsto \tilde{f} $ and it is  invariant under the operator $\tilde{T}$.
 Since $ T $ commutes with rotations, i.e. $T\varrho (k)= \varrho (k)T $ we see that
 $$\tilde{T} \tilde{f}(x,k) = T(\varrho (k)f)(x) = \varrho (k) (Tf)(x) = (Tf)(k x).$$
 The boundedness of $ \tilde{T} $ on $ L^p(\R^N, \mathcal{H}, w(|x|)dx) $ gives
 \begin{eqnarray*}
  & & \int_{\R^N} \Big ( \int _K | T f(k x)|^2 dk \Big )^{\frac{p}{2}}w(|x|)dx
  \leq  C \int_{\R^N} \Big ( \int _K | f(k
 x)|^2 dk \Big )^{\frac{p}{2}}w(|x|)dx
 \end{eqnarray*}
 which translates into the boundedness of the restriction of  $\tilde{T}$
 to the weighted mixed norm space $L^{p,2}_{w}(\R^N)$. This
 proves our claim.\\

%----------------------------------------------------------------------------
\subsection{The point spectrum of $\Delta_\K $ on $ L^p_{\rho,\K}(\R^n)$ }
 In this subsection we precisely determine the point spectrum of
 the Dunkl-Laplacian on the weighted spaces $ L^p_{\rho,\K}(\R^n).$
 In the unweighted case, the spectrum of $ \Delta_\K $ turns out to
 be the half line $ [0,~\infty) $ for all $ p.$ This follows from a
 multiplier theorem for the Dunkl transform proved in \cite{DW}. On
 the other hand we do not have a multiplier theorem on the weighted
 spaces $ L^p_{\rho,\K}(\R^n)$ for $ p \neq 2.$ However, it is not
 difficult to determine the  point-spectrum $
 \sigma_{pt}(\Delta_\K+\rho^2) $ on these spaces.

 \begin{thm}\label{T16} For any $ 1 \leq p < \infty $ we have
 $ \sigma_{pt}(\Delta_\K +\rho^2)  = \mathfrak{P}^0_p,$ the
 interior of $\mathfrak{P}_p.$
 \end{thm}
 \begin{proof}
 It is enough to prove that if $f$ is an eigenfunction of  $\Dk$ in $
 L^p_{\rho,\K}(\R^n)$, $1\leq p< \infty $ with eigenvalue $\la^2$, then  $ |\Im(\la)|<
 \gm_p\rho$. If $ f $ is such an eigenfunction, then
 $$
 \int_{\sn}\Dk f(r\om)Y_{m,j}^h(\om)h_{\K}^2(\om)d\si(\om)=
  \la^2\int_{\sn} f(r\om)Y_{m,j}^h(\om)h_{\K}^2(\om)d\si(\om) $$
  for $m=0,1,2,\cdots $, $j=1,2,\cdots , d(m)$. Here  $Y_{m,j}^h$ is a spherical h-harmonic of degree $m$
  taken from the orthonormal basis $\{Y_{m,j}^h: \; j=1,2,\ldots ,d(m),\; m = 0,1,2,\ldots  \} $
  for $L^2(\sn, h^2_{\K}(\om)d\si(\om))$.  The Dunkl-Laplacian has the explicit form (see page.159 in \cite{DX})
  $$\Dk=-\Big (\frac{d^2}{dr^2}+\frac{N-1}{r}\frac{d}{dr}+\frac{1}{r^2}\D_{h,0}\Big )$$
  where $\D_{h,0}$ is the spherical part of $\Dk $ and $N=n+2\gm.$ In view of this, we have
  \begin{eqnarray*}
    \left(\frac{d^2}{dr^2}+\frac{N-1}{r}\frac{d}{dr} +\lambda^2 \right)f_{m,j}(r) =  -\frac{1}{r^2}
  \int_{\sn}\D_{h,0} f(r\om)Y_{m,j}^h(\om)h_{\K}^2(\om)d\si(\om)
     \end{eqnarray*}
  where
  $$ f_{m,j}(r) = \int_{\sn} f(r\om)Y_{m,j}^h(\om)h_{\K}^2(\om)d\si(\om) $$
  are the spherical harmonic coefficients of $ f.$\\

  Making use of the facts that $\D_{h,0} $ is selfadjoint on $L^2(\sn, h^2_{\K}(\om)d\si(\om))$
  and $Y_{m,j}^h$ are
  eigenfunctions of $\D_{h,0} $ with eigenvalues $-m(m+N-2)$, we see that
   the functions $f_{m,j}$ satisfy the differential
  equation
  $$\left(r^2\frac{d^2}{dr^2}+(N-1)r\frac{d}{dr}+\la^2r^2-m(m+N-2)
  \right)f_{m,j}(r)=0 $$
  for $r>0$.  If we define  $g_{m,j}(r)=r^{N/2-1}f_{m,j}(r)$, then these functions
  satisfy the Bessel differential equation of type
  $(m+\frac{N}{2}-1)$, i.e.
  $$\left(r^2\frac{d^2}{dr^2}+r\frac{d}{dr}+\la^2r^2-(m+\frac{N}{2}-1)^2
  \right)g_{m,j}(r)=0 $$
  for $ r>0 $. We know that the linearly independent solutions of Bessel differential
  equation of type $\nu $ are given by the Bessel functions of first kind $J_\nu(\la r) $ and second kind
  $Y_\nu(\la r),$ see e.g. \cite{NU}, page 219. Therefore,
  $$g_{m,j}(r)=C_1(\lambda) J_{m+N/2-1}(\la r)+C_2(\lambda)Y_{m+N/2-1}(\la r).$$
  In the above  $C_2(\lambda) $ has to be $ 0 $  since $g_{m,j}(r) $ and $J_{m+N/2-1}(\la r) $ are locally integrable
  functions whereas  $Y_{m+N/2-1}(\la r) $ is not  locally integrable near the origin.
  This follows from the fact that $$Y_{m+N/2-1}(x)\approx - \frac{2^{m+N/2-1}\Gamma(m+N/2-1)}{\pi x^{m+N/2-1}}~$$
  for $~ x\rightarrow 0 $. Thus $g_{m,j}(r)=r^{N/2-1}f_{m,j}(r)=C_1(\lambda)J_{m+N/2-1}(\la r)$
  and hence
  $$f_{m,j}(r)= C_1 \frac{J_{m+N/2-1}(\la r)}{(\lambda r)^{N/2-1}} $$ for $r>0
  $. Since $f \in L^p_{\rho, \K}(\R^n)$, it can be easily seen
  that
  $$\int_0^{\infty}|f_{m,j}(r)|^p(K_{n/2+\gm}(\rho r))^{p\gm_p} r^{N-1}dr<\infty $$
  which implies that
  $$\int_0^{\infty}\left| \frac{J_{m+N/2-1}(\la r)}{(\lambda r)^{N/2-1}}\right|^p(1+ r)^{\frac{(n+2\gm-1)}{2}p \gm_p}
     e^{-\rho p\gm_p r} r^{N-1}dr<\infty. $$
  This is possible only if $|\Im(\la)|< \gm_p \rho $ which proves
  Theorem \ref{T16}.
\end{proof}

%----------------------------------------------------------------------------

\subsection{The chaotic behavior of the Dunkl heat semigroup:}
 In this subsection we prove the remaining main theorems, namely
 Theorem \ref{T5} and Theorem \ref{T6} regarding the chaotic
 behavior of the semigroup $T_t^c = e^{ct}e^{-t(\Dk+\rho^2)}$,
 $c \in \R$ on the spaces $L_{\rho,\K}^p(\R^n)$ and
 $L^{p,2}_{\rho, \K}(\R_n)$. For the space $L^{p,2}_{\rho,
 \K}(\R_n)$, we assume that $2\gm$ is an integer to give a
 complete picture of the chaotic behavior of the semigroup
 $T_t^c = e^{ct}e^{-t(\Dk+\rho^2)}$, $c \in \R$ on that space. We
 prove these results by imitating the proofs presented in the Subsection 2.3 where we have
 discussed the chaotic behavior of heat semigroup in the Euclidean set up.
 So we give a very sketchy outline of  these proofs. We use all the notations
 introduced in the Subsection 2.3 with some appropriate changes required for the Dunkl set up.
 For example, we have to replace  the Euclidean translations of the spherical
 functions $\tau_x \varphi_{\la}$ by the Dunkl translations of the
 Bessel functions $\tau_x \varphi_{\la, \K}$ which are defined
  by  $$\tau_x\varphi_{\la,\K}(y)= \int_{\sn} E_\K(ix,\la\om)E_{\K}(iy,\la\om)h_\K^2(\om)d\si(\om)$$
  where $\varphi_{\la, \K}$ are Dunkl spherical functions defined
  in  Subsection 3.2. All other notations $\Lambda_p, ~\om_c, ~c_p, ~A_j's,
  ~\mathcal{A}_j's$, $\B_0, ~\B_{\infty}$ and  $\B_{Per}$  are the same
  as in Subsection 2.3. Note that $\mathcal{A}_j$'s are defined
  by using Dunkl translations of Dunkl spherical functions
  $\tau_x \varphi_{\la, \K}$ in place of Euclidean translations
  of spherical functions $\tau_x \varphi_{\la}$ and
  $\B_0, \B_{\infty}$ and  $\B_{Per}$ are defined for the space
  $\B=L^p_{\rho, \K}(\R^n)$ and the semigroup $T_t^c= e^{ct}e^{-t(\Dk+\rho^2-c)}$.
  Also note that  $\tau_x\varphi_{\la,\K} \in L^p_{\rho, \K}(\R^n)$
  whenever $\la \in \Lambda_p^0$ and $p\neq 2$. This follows from the estimate
  $|E_{\K}(ix, \la \om)|\leq C e^{|\Im(\la)|\:|x|}$. See
  \cite{DJ}, where estimates of partial derivatives of $E_{\K}(x,z)$, $x \in \R^n,\; z \in \C^n$
  are given. It is then clear that $\mathcal{A}_j\subset L^p_{\rho, \K}(\R^n)$
  and also for $f \in \mathcal{A}_j$, $T_tf = e^{-t\om_c(\la)}f $ for each $1\leq j \leq 3 $.\\

  With these notations, we have the following proposition which is the analogue of Proposition \ref{P1}.

 \begin{prop}\label{P6} For each $1\leq j \leq 3 $,  $\mathcal{A}_j$ is dense in $L^p_{\rho,\K}(\R^n)$,
 $1 \leq  p<\infty,~~ p\neq 2 $ and $span(\mathcal{A}_1)\subset
 \B_0$, $span(\mathcal{A}_2)\subset \B_\infty$ and
 $span(\mathcal{A}_3)\subset \B_{Per}$ provided $c>c_p$.
\end{prop}
\begin{proof}
The set inclusions $span(\mathcal{A}_1)\subset \B_0$,
$span(\mathcal{A}_2)\subset \B_\infty$ and
$span(\mathcal{A}_3)\subset \B_{Per}$ can be proved by the same
arguments given in the proof of Proposition \ref{P1}. Now we need
to prove the densities of $\B_0,\; \B_\infty $ and $\B_{Per}$
which will follow once we prove that of $\mathcal{A}_j$'s are
dense in $L^p_{\rho, \K}(\R^n)$. Suppose the span of
$\mathcal{A}_1$ is not dense in $L^p_{\rho, \K}(\R^n)$, $1<
p<\infty $. As the dual of $L^p_{\rho, \K}(\R^n) $ can be
identified with $L^{p'}(\R^n, (\varphi_{i\rho, \K}(x))^{p'\gm{p'}}
h_\K(x)^2 dx)$ (where $1/p+1/p'=1$ and $1<p'<\infty$), there
exists $g\in L^{p'}(\R^n, (\varphi_{i\rho, \K}(x))^{p'\gm{p'}}
h_\K(x)^2dx) $ such that
$$\int_{\R^n}g(y)\tau_x\varphi_{\la}(y)h_\K^2(y)dy =0$$
for all $\la \in A_1 $ and $x \in \R^n $. Note that $L^{p'}(\R^n,
(\varphi_{i\rho, \K}(x))^{p'\gm{p'}}h_\K(x)^2 dx)$ is a subspace of $
 L^1(\R^n, h_\K^2(x) dx)$ for $1<p'<\infty$ and hence the
map
$$\la \mapsto \int_{\R^n}g(y)\tau_x\varphi_{\la}(y)h_\K^2(y)dy $$
is a continuous function on $\Lambda_p^0$. Moreover, by Morera and
Fubini, the map is holomorphic. Since $A_1$ is a nonempty open
subset of $\Lambda_p^0$ it follows that
$$\int_{\R^n}g(y)\tau_x\varphi_{\la}(y)h_\K^2(y)dy=0 $$
for all $\la \in \Lambda_p$; in particular, for all $\la \in \R$.
In view of the definition of $\tau_x\varphi_\la$ and Lemma 3.1
(1), we infer that for any $t>0$
$$e^{-t\Dk}g(x)= c_N \int_0^\infty e^{-t\la^2}\Big ( \int_{\R^n}
\tau_{-x}\varphi_{\la}(y)g(y)h_{\K}^2(y)dy \Big )
\la^{n+2\gm-1}d\la =0.$$ Since the heat semigroup $e^{-t\Dk}$ is
strongly continuous on $L^1(\R^n, h_\K^2 dx)$, $ g \in L^1(\R^n,
h_\K^2 dx)$, $e^{-t\Dk}g=0$ for every $t>0$ implies $g=0$. When $
p = 1 $ we have a bounded function $ g_1 $ such that
$$\int_{\R^n}g_1(y)\tau_x\varphi_{\la, \K}(y) \tilde{K}_{n/2+\gm}(y) dy=0 $$
for all $\la \in \Lambda_1.$  Since  $ g(y) = g_1(y)
\tilde{K}_{n/2+\gm}(y) $ belongs to $ L^1(\R^n, h_{\K}^2(x)dx) $
we can conclude that $ g_1 = 0 $ as before. This proves the
density of span of $ \mathcal{A}_1$. The density of the spans of
$\mathcal{A}_2$ and $\mathcal{A}_3$ are similarly proved.
\end{proof}

\begin{rem}\label{R2}
 By keen observation of the above proof, Proposition \ref{P6} still holds if we replace the space $L^p_{\rho, \K}(\R^n)$ by
 the weighted mixed norm space $L^{p,2}_{\rho, \K}(\R^n)$ for $1< p<\infty $, $p\neq 2 $. To see this we only have to
 check that $\mathcal{A}_j $ are subsets of $L^{p,2}_{\rho, \K}(\R^n)$ for $1\leq j \leq 3$ and the
 strong continuity of heat semigroup $e^{-t\Dk}$ on $L^{p',2}_{\rho, \K}(\R^n) $. The latter fact
 has been already proved in Theorem \ref{T14}. To see the inclusions, we have $\tau_x\varphi_\la\in L^{p,2}_{\rho,
 \K}(\R^n)$ whenever $\la \in \Lambda_p^0$.
 This also follows from the estimate $|E_{\K}(ix, \la \om)|\leq C
 e^{|\Im(\la)|\:|x|}$. It is then clear that $\mathcal{A}_j\subset L^{p,2}_{\rho, \K}(\R^n).$
\end{rem}

 Now we are ready to prove the remaining main theorems (Theorem \ref{T5} and Theorem \ref{T6}).\\

\noindent {\bf Proof of the Theorem \ref{T5}:}
 In Theorem \ref{T13} we already proved that $T_t = e^{-t{A}}$ defines
 a strongly continuous semigroup on $L^p_{\rho, \K}(\R^n)$ for $1\leq p < \infty.$
 In view of  Corollary \ref{C1} and Proposition \ref{P6}, $T_t^c$ is chaotic on
 $L^p_{\rho, \K}(\R^n)$ for $c>c_p$ and $1\leq p<\infty,~p\neq 2 $.
 This proves the sufficient part of part (1) of Theorem \ref{T5}. For necessary part
 we make use  of Theorem \ref{T16} according to which the point spectrum $ \si_{pt}(\Dk+\rho^2-c) $ of the operator
 $(\Dk+\rho^2-c) $ on $L^p_{\rho,\K}(\R^n) $ is given by
 $$ \mathfrak{P}^0_p -c =\{ \la^2+\rho^2-c: | \Im(\la)|< \gm_p \rho \}$$  for $1\leq p < \infty .$
 By the geometric form of the above set, it can be easily seen that
 the set $\si_{pt}(\D+\rho^2-c)\cap i\R  $ is empty for $c\leq c_p$ and
 hence in view of Theorem \ref{T1}, $T_t^c$ is not chaotic. This
 proves part (1) of Theorem \ref{T5}. Now we proceed to part (2).
 We note that for any $c\in \R$, the semigroup $T_t^c=e^{-t(A-c)}$ is
 not hypercyclic on $L^{\infty}_{\rho, \K}(\R^n)$, since
 for any $f \in L^{\infty}_{\rho, \K}(\R^n)$, $T_t^c f $ is a continuous bounded
 function and hence the closure of the orbit $\{ T_t^cf: t>0 \} $ in
 $L^{\infty}_{\rho, \K}(\R^n)$ is a subset of the subspace of all
 continuous bounded functions which is a strictly smaller than $L^{\infty}_{\rho,
 \K}(\R^n)$. This proves part (2). For part (3), we know that the $L^2-$spectrum
 $\si_2(\Dk) $ of $\Dk $ on  $ L^2_{\rho,\K}(\R^n) = L^2(\R^n, h_\K(x)^2 dx)
 $ is given by $[0, ~\infty) $ and hence $\si_2(\Dk+\rho^2-c)\cap i\R
 $ has at most one point.  In view of Theorem \ref{T1}, $T_t^c $ is
 not chaotic on $L^2_{\rho,\K}(\R^n)$. This proves part (3)
 which completes the proof of Theorem \ref{T5}.\\

\noindent {\bf Proof of the Theorem \ref{T5.1}:}
Let us define  $a_p := \frac{2\rho^2}{p'}$ for $1\leq p < 2$  and  $a_p:=  \frac{2\rho^2}{p} $ for $p>2$. In view of Theorem \ref{T13}, for any $f \in
 L^p_{\rho,\K}(\R^n)$, we have
 $$\|T_t^cf\|_{L^p_{\rho,\K}(\R^n)}=e^{ct}\|T_tf\|_{L^p_{\rho,\K}(\R^n)} \leq C t^{\frac{n+2\gm-1}{2}\gm_p}
 e^{-t(a_p-c)}\|f\|_{L^p_{\rho,\K}(\R^n)}$$
 where $T_t = e^{-tA}$. From the above, the operators $T^c_t$ are
 bounded uniformly in $t$ on $L^p_{\rho,\K}(\R^n)$ for $c<a_p$.
 Consequently, for each $f\in L^p_{\rho,\K}(\R^n)$
 the orbit $\{ T_t^cf: t>0 \} $ is a bounded set and hence it
 cannot be dense in $L^p_{\rho,\K}(\R^n)$. Thus $T_t^c$ fails
 to be hypercyclic. This proves part (1) and part (2)
 which completes the proof of Theorem \ref{T5.1}.\\

 \noindent
 {\bf Proof of the Theorem \ref{T6}:} In view of Theorem \ref{T14}, the
 semigroup $T^c_t$ is strongly continuous on $L^{p,2}_{\rho,
 \K}(\R^n)$ for $1<p<\infty, ~ p\neq 2.$  In view of Remark
 \ref{R2} and Corollary \ref{C1}, $T^c_t$ is chaotic on the space $L^{p,2}_{\rho,
 \K}(\R^n)$ for $c>c_p$ and $1\leq p<\infty, ~ p\neq 2 $. This proves
 the part (1). For part (2), in view of Theorem \ref{T14}, for any $f \in L^{p,2}_{\rho,
 \K}(\R^n)$ we have
 $$\|T^c_tf\|_{L^{p,2}_{\rho,\K}(\R^n)}\leq C t^{\frac{n+2\gm-1}{2}\gm_p}
 e^{-t(c_p-c)} \|f\|_{L^{p,2}_{\rho,\K}(\R^n)}.$$
 For $c<c_p$, by the same arguments given in the proof part (3) of Theorem \ref{T5},
 $T_t^c$ fails to be hypercyclic. This proves part (2)
 which completes the proof of Theorem \ref{T6}.\\

%--------------------------------------------------------------------------

\begin{center}
{\bf Acknowledgments}
\end{center}
The first author is thankful to CSIR, India, for
the financial support. The work of the second author is supported
by J. C. Bose Fellowship from the Department of Science and
Technology (DST) and also by a grant from UGC via DSA-SAP. Both authors wish
to thank Rudra Sarkar for some useful conversations regarding the subject matter of this article.

%-------------------------------------------------------------------------

\end{document}